\documentclass[12pt, reqno]{amsart}

\usepackage{amsmath, amssymb, amscd, amsthm, comment, amsxtra}
\usepackage{graphicx}
\usepackage{subfigure}
\usepackage{enumitem}
\usepackage[pdftex, linktocpage=true]{hyperref}
\usepackage{euscript}
\usepackage[format=plain,labelfont=bf,up,width=.99\textwidth]{caption}
\usepackage[toc,page]{appendix}

\theoremstyle{plain}
\newtheorem{thm}{Theorem}[section]
\newtheorem*{thma}{Non-rigidity Theorem}
\newtheorem{lem}[thm]{Lemma}
\newtheorem{prop}[thm]{Proposition}
\newtheorem{cor}[thm]{Corollary}

\theoremstyle{definition}
\newtheorem{defn}[thm]{Definition}

\newtheorem{notn}[thm]{Notation}

\title[Renormalization in the Golden-Mean Semi-Siegel H\'enon Family]{Renormalization in the Golden-Mean Semi-Siegel H\'enon Family: Universality and Non-rigidity}
\author{Jonguk Yang}

\begin{document}

\begin{abstract}
It was recently shown in \cite{GaYa2} that appropriately defined renormalizations of a sufficiently dissipative golden-mean semi-Siegel H\'enon map converge super-exponentially fast to a one-dimensional renormalization fixed point. In this paper, we show that the asymptotic two-dimensional form of these renormalizations is {\it universal}, and is parameterized by the average Jacobian. This is similar to the limit behavior of period-doubling renormalizations in the H\'enon family considered in \cite{dCLM}. As an application of our result, we prove that the boundary of the golden-mean Siegel disk of a dissipative H\'enon map is non-rigid.

\end{abstract}

\maketitle

\section{Introduction}\label{sec:introduction}

The archetypical class of examples in holomorphic dynamics is given by the {\it quadratic family}:
\begin{displaymath}
f_c(z) = z^2 +c \hspace{5mm} \text{for } c \in \mathbb{C}.
\end{displaymath}
Despite its apparent simplicity, the dynamics of this family is incredibly rich, and exhibits many of the key features that are observed in the general case. In dynamics of several complex variables, the role of the quadratic family is assumed by its two-dimensional extension:
\begin{displaymath}
H_{c,b}(x,y)=(x^2+c - by,x) \hspace{5mm} \text{for } c \in \mathbb{C} \text{ and } b \in \mathbb{C} \setminus \{0\}
\end{displaymath}
called the {\it (complex quadratic) H\'enon family}.

Since we have
\begin{displaymath}
H_{c,b}^{-1}(x,y)=\left(y, \frac{y^2 + c - x}{b}\right),
\end{displaymath}
a H\'enon map $H_{c,b}$ is a polynomial automorphism of $\mathbb{C}^2$. Moreover, it is easy to see that $H_{c,b}$ has constant Jacobian:
\begin{displaymath}
\text{Jac} \, H_{c,b} \equiv b.
\end{displaymath}
Note that for $b=0$, the map $H_{c,b}$ degenerates to the following embedding of $f_c$:
\begin{displaymath}
(x,y) \mapsto (f_c(x),x).
\end{displaymath}
Hence, the parameter $b$ determines how far $H_{c,b}$ is from being a degenerate one-dimensional system. In this paper, we will always assume that $H_{c,b}$ is a dissipative map (i.e. $|b|<1$). 

\begin{figure}[h]
\centering
\includegraphics[scale=0.4]{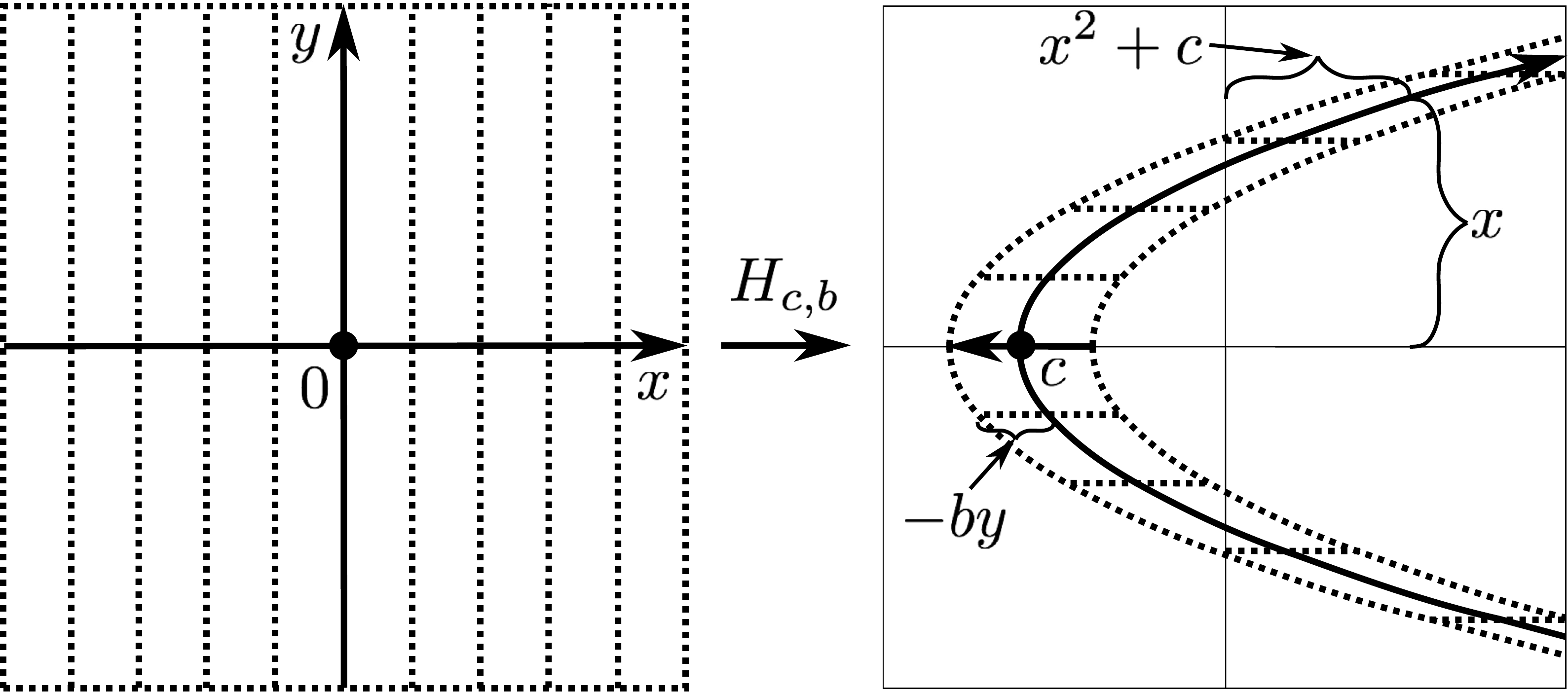}
\caption{A H\'enon map $H_{c,b}$. Note that vertical lines are scaled uniformly by $-b$, and then mapped to horizontal lines.}
\label{fig:henon}
\end{figure}

A H{\'e}non map $H_{c,b}$ is determined uniquely by the multipliers $\mu$ and $\nu$ at a fixed point $\mathbf{p}$. In particular, we have
\begin{displaymath}
b = \mu\nu,
\end{displaymath}
and
\begin{displaymath}
c=(1+\mu\nu)\left(\frac{\mu}{2}+\frac{\nu}{2}\right)-\left(\frac{\mu}{2}+\frac{\nu}{2}\right)^{2}.
\end{displaymath}
When convenient, we will write $H_{\mu,\nu}$ instead of $H_{c,b}$ to denote a H\'enon map.

Suppose that one of the multipliers, say $\mu$, is irrationally indifferent, so that
\begin{displaymath}
\mu=e^{2\pi i \theta}
\hspace{5mm} \text{for some} \hspace{5mm}
\theta \in (0,1) \setminus \mathbb{Q}.
\end{displaymath}
Then we have
\begin{displaymath}
|b| = |\nu|.
\end{displaymath}
In this case, the H\'enon map $H_{\mu, \nu}$ is said to be {\it semi-Siegel} if there exist neighborhoods $N$ of $(0,0)$ and $\mathcal{N}$ of $\mathbf{p}$, and a biholomorphic change of coordinates
\begin{displaymath}
\phi : (N, (0,0)) \to (\mathcal{N}, \mathbf{p})
\end{displaymath}
such that 
\begin{displaymath}
H_{\mu, \nu} \circ \phi = \phi \circ L,
\end{displaymath}
where $L(x,y):=(\mu x,\nu y)$. A classic theorem of Siegel states, in particular, that $H_{\mu,\nu}$ is semi-Siegel whenever $\theta$ is {\it Diophantine}. That is, for some constants $C$ and $d$, we have
\begin{displaymath}
q_{n+1}< C q_n^d,
\end{displaymath}
where $p_n/q_n$ are the continued fraction convergents of $\theta$ (see Section \ref{sec:motivation}). In this case, the linearizing map $\phi$ can be biholomorphically extended to
\begin{displaymath}
\phi: (\mathbb{D} \times \mathbb{C}, (0,0)) \rightarrow (\mathcal{C}, \mathbf{p})
\end{displaymath}
so that its image $\mathcal{C} := \phi(\mathbb{D} \times\mathbb{C})$ is {\it maximal} (see \cite{MNTU}). We call $\mathcal{C}$ the {\it Siegel cylinder} of $H_{\mu, \nu}$. In the interior of $\mathcal{C}$, the dynamics of $H_{\mu, \nu}$ is conjugate to rotation by $\theta$ in one direction, and compression by $\nu$ in the other direction. Clearly, the orbit of every point in $\mathcal{C}$ converges to the analytic disk $\mathcal{D}:=\phi(\mathbb{D} \times\{0\})$ at height $0$. We call $\mathcal{D}$ the {\it Siegel disk} of $H_{\mu, \nu}$.

\begin{figure}[h]
\centering
\includegraphics[scale=0.4]{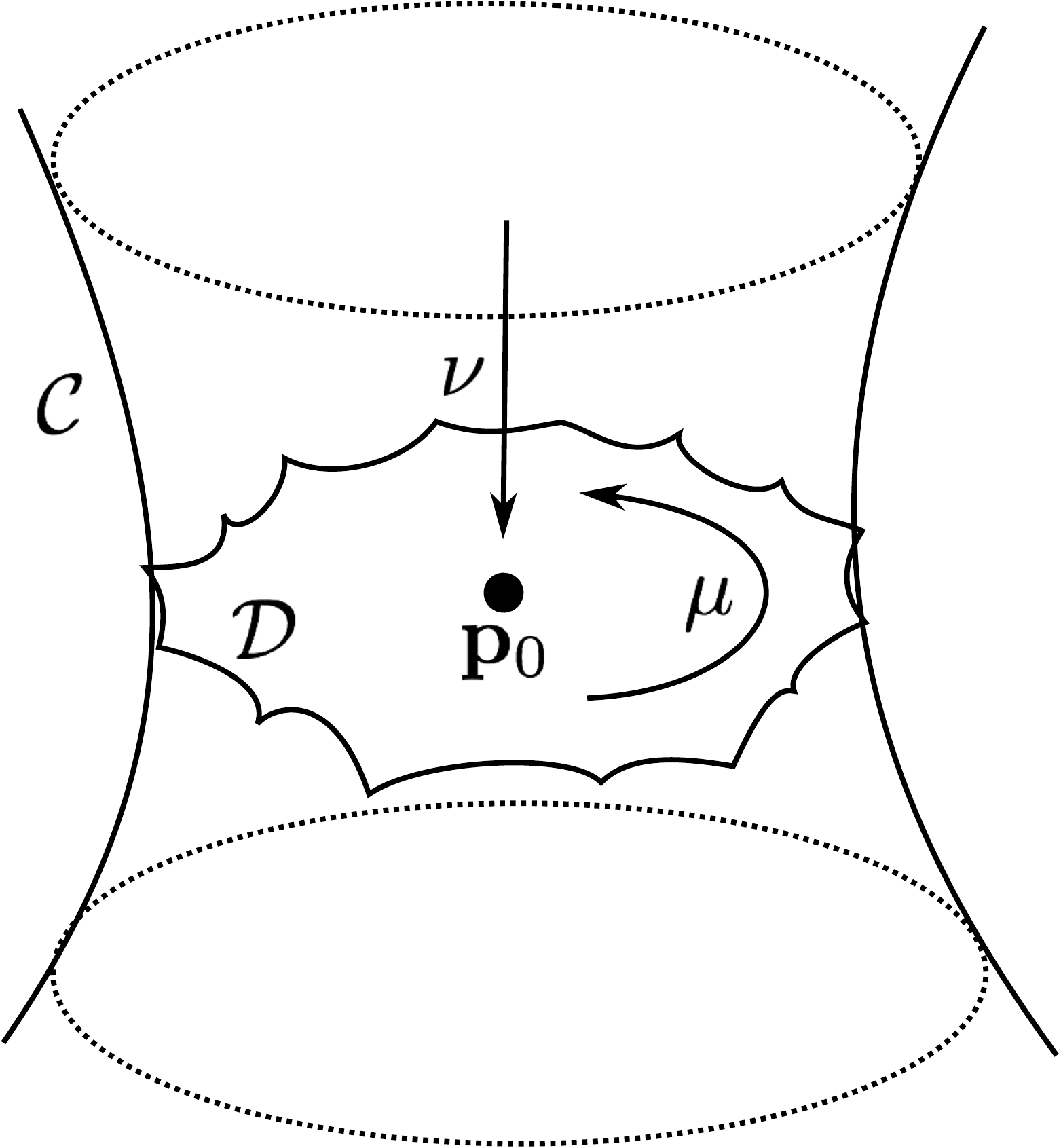}
\caption{The Siegel cylinder $\mathcal{C}$ and the Siegel disk $\mathcal{D}$ of $H_{\mu, \nu}$.}
\label{fig:siegelcylinder}
\end{figure}

The geometry of Siegel disks in one dimension is a challenging and important topic that has been studied by numerous authors; including Herman \cite{He}, McMullen \cite{Mc}, Petersen \cite{P}, Inou and Shishikura \cite{ISh}, Yampolsky \cite{Ya}, and others. In the two-dimensional H\'enon family, the corresponding problems have been wide open until a very recent work of Gaidashev, Radu, and Yampolsky \cite{GaRYa}, who proved:

\begin{thm}[Gaidashev, Radu, Yampolsky]\label{continuous circle}
Let $\theta_*=(\sqrt{5}-1)/2$ be the inverse golden mean, and let $\mu_*=e^{2\pi i\theta_*}$. Then there exists $\epsilon>0$ such that if $|\nu|<\epsilon$, then the boundary of the Siegel disk $\mathcal{D}$ of $H_{\mu_*,\nu}$ is a homeomorphic image of the circle. In fact, the linearizing map
\begin{displaymath}
\phi:\mathbb{D} \times \{0\}\rightarrow \mathcal{D}
\end{displaymath}
extends continuously and injectively (but not smoothly) to the boundary.
\end{thm}

In the author's joint paper with Yampolsky \cite{YaY}, we have obtained the first geometric result about Siegel disks in the H\'enon family:

\begin{thm}[Yampolsky, Y.]\label{nonsmooth circle}
Let $\mu_*$ and $\epsilon>0$ be as in Theorem \ref{continuous circle}. Then for $|\nu|<\epsilon$ the boundary of the Siegel disk $\mathcal{D}$ of $H_{\mu_*,\nu}$ is not $C^1$-smooth.
\end{thm}

The proofs of Theorem \ref{continuous circle} and \ref{nonsmooth circle} are based on the renormalization theory developed by Gaidashev and Yampolsky in \cite{GaYa}. Generally speaking, a renormalization of a dynamical system is defined as a rescaled first return map on an appropriately chosen subset of the phase space. In their paper, Gaidashev and Yampolsky considered the semi-Siegel H\'enon maps within the context of a Banach space $\mathcal{B}_2$ of dynamical systems called {\it almost commuting pairs}. They then formulated renormalization as an operator $\mathbf{R}_{\text{GY}}$ from $\mathcal{B}_2$ to itself. They were able to show that this operator is analytic, and that it has a hyperbolic fixed point $\Sigma_* \in \mathcal{B}_2$. In \cite{GaRYa}, they went on to prove that the stable manifold of $\Sigma_*$ does indeed contain the almost commuting pairs that correspond to sufficiently dissipative semi-Siegel H\'enon maps of the golden-mean type. 

It is important to note that the fixed point $\Sigma_*$ for $\mathbf{R}_{\text{GY}}$ is a degenerate one-dimensional system. Hence, when the renormalization sequence of an almost commuting pair $\Sigma$ converges to $\Sigma_*$, it loses its dependence on the second variable along the way. In fact, Gaidashev and Yampolsky showed that this must happen at a super-exponential rate.

In this paper, we describe the behaviour of almost commuting pairs as they approach the space of degenerate one-dimensional systems under renormalization. For this purpose, we adopt a new renormalization operator $\mathbf{R}$ that we obtain by modifying the construction of $\mathbf{R}_{\text{GY}}$. The main difference between these two operators is that while $\mathbf{R}_{\text{GY}}$ is based on a diagonal embedding of the pairs of one-dimensional maps:
\begin{displaymath}
(\eta, \xi) \mapsto \left(\begin{bmatrix}
\eta \\
\eta
\end{bmatrix}, 
\begin{bmatrix}
\xi \\
\xi
\end{bmatrix}\right),
\end{displaymath}
the operator $\mathbf{R}$ is based on a H\'enon-like embedding (see \eqref{eq:embedding}). Although the former embedding has the benefit of being more symmetric, the latter embedding allows us to track two-dimensional deviations from its image more precisely and more explicitly. However, it should be noted that $\mathbf{R}$ and $\mathbf{R}_\text{GY}$ are still related closely enough that a number of proofs given in \cite{GaRYa} can be directly transferred to our setting, {\it mutatis mutandis} (in particular, see Theorem \ref{continuity} and \ref{siegel henon}).

The central result of this paper is that in the limit of renormalization, the almost commuting pairs take on a {\it universal} two-dimensional shape as they flatten into degenerate one-dimensional systems. This statement is formulated explicitly in Theorem \ref{universality}. The proof relies on an analysis of the average Jacobian of almost commuting pairs on their invariant renormalization arcs. A similar approach was taken by de Carvalho, Lyubich, and Martens in \cite{dCLM} to study the limits of period-doubling renormalization in the H\'enon family.

The universality phenomenon described in Theorem \ref{universality} has deep consequences on the geometry of the golden-mean Siegel disk of dissipative H\'enon maps. In \cite{dCLM}, de Carvalho, Lyubich, and Martens used universality to show, in particular, that the invariant Cantor set for period-doubling renormalization is non-rigid. In this paper, we are able to obtain the following analogous result:

\begin{thma}
Let $\mu_*$ and $\epsilon>0$ be as in Theorem \ref{continuous circle}. If $|\nu_1|, |\nu_2|<\epsilon$, and $|\nu_1| \neq |\nu_2|$, then the two semi-Siegel H\'enon maps $H_{\mu_*, \nu_1}$ and $H_{\mu_*, \nu_2}$ cannot be $C^1$ conjugate on the boundary of their respective Siegel disks.
\end{thma}

Non-rigidity is the first known property of Siegel disks of H\'enon maps that is unique to higher dimensions. In the one-dimensional case, McMullen showed that two quadratic-like maps with a Siegel disk of the same bounded type rotation number are $C^{1+\alpha}$ conjugate on their Siegel boundary (see \cite{Mc}).

\section{Motivation}\label{sec:motivation}

Consider the quadratic polynomial
\begin{displaymath}
f_{c_0}(z) = z^2 + c_0
\end{displaymath}
that has a fixed Siegel disc $\mathcal{D}_0 \subset \mathbb{C}$ with rotation number $\theta \in \mathbb{R} /\mathbb{Z}$. Since $\theta$ must be irrational, it is represented by an infinite continued fraction:
\begin{displaymath}
\theta = [a_0, a_1, \ldots{}] = \cfrac{1}{a_0+\cfrac{1}{a_1+\cfrac{1}{a_2 + \ldots{}}}}.
\end{displaymath}
The {\it $n$th partial convergent of $\theta$} is the rational number
\begin{displaymath}
\frac{p_n}{q_n} = [a_0, a_1, \ldots , a_n].
\end{displaymath}
The denominator $q_n$ is called the $n$th {\it closest return moment}. The sequence $\{q_n\}_{n=0}^\infty$ satisfy the following inductive relation:
\begin{equation}\label{closest return induction}
q_0 = 1
\hspace{5mm} , \hspace{5mm}
q_1 =a_0
\hspace{5mm} \text{and} \hspace{5mm}
q_{n+1} = a_n q_n + q_{n-1}
\hspace{5mm} \text{for} \hspace{5mm}
n \geq 1.
\end{equation}

We say that $\theta$ is of {\it bounded type} if $a_n$'s are uniformly bounded. The simplest example of a bounded type rotation number is the inverse golden mean:
\begin{displaymath}
\theta_* = \frac{\sqrt{5}-1}{2} = [1, 1, \ldots].
\end{displaymath}

The following theorem is due to Petersen \cite{P}.

\begin{thm}\label{siegel boundary rotation}
Let
\begin{displaymath}
f_{c_0}(z) := z^2 + c_0
\end{displaymath}
be a quadratic polynomial that has a fixed Siegel disk $\mathcal{D}_0 \subset \mathbb{C}$ with bounded type rotation number $\theta \in \mathbb{R}/\mathbb{Z}$. Then $f_{c_0}$ has its critical point $0$ on its Siegel boundary $\partial \mathcal{D}_0$, and the restriction $f_{c_0}|_{\partial \mathcal{D}_0} : \partial \mathcal{D}_0 \to \partial \mathcal{D}_0$ is quasi-symmetrically conjugate to the rigid rotation of the circle $S^1$ by angle $\theta$.
\end{thm}

\begin{figure}[h]
\centering
\includegraphics[scale=0.6]{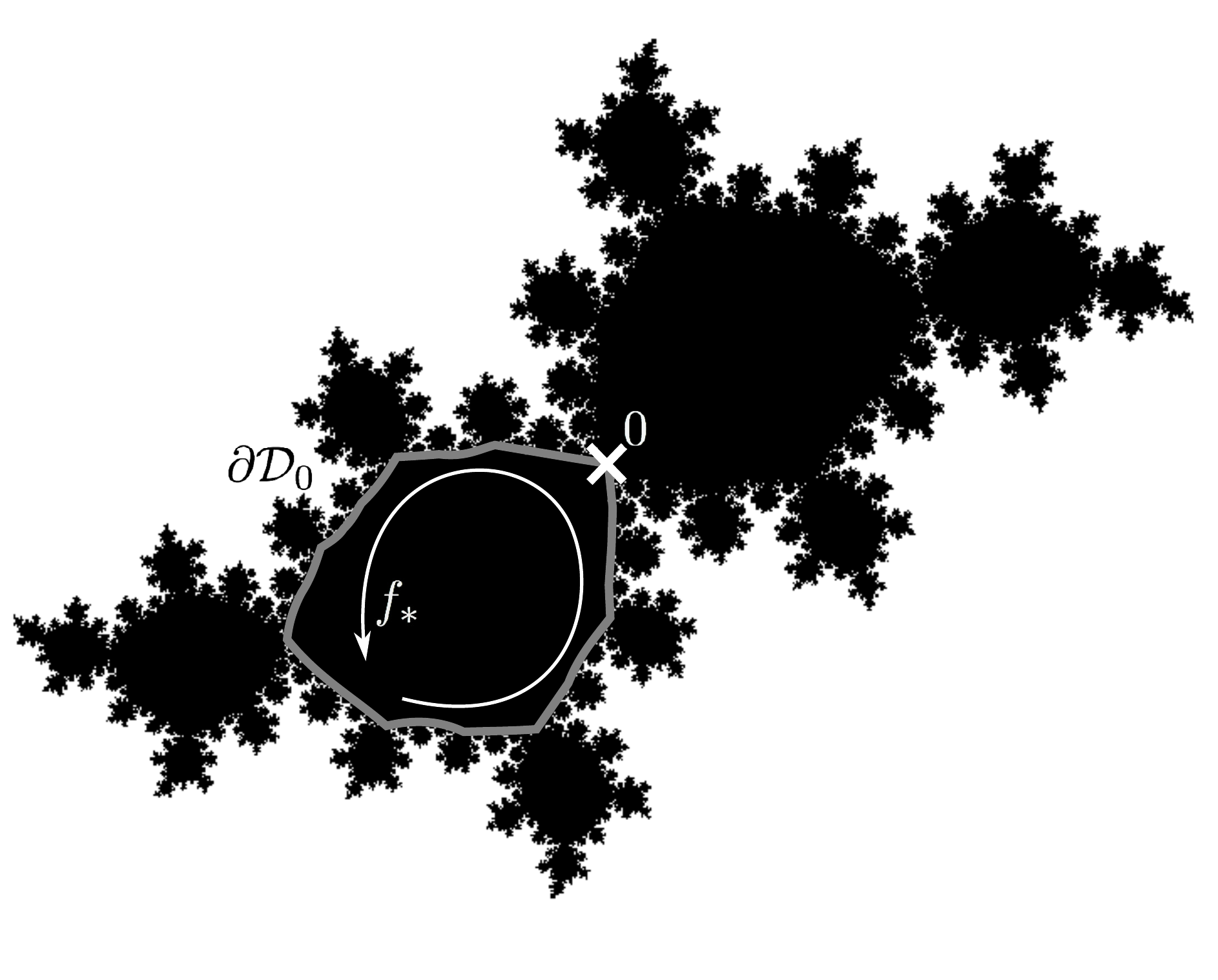}
\caption{The Siegel boundary $\partial \mathcal{D}_0$ of the golden-mean Siegel quadratic polynomial $f_*$. The critical point $0$ is on $\partial \mathcal{D}_0$.}
\end{figure}

Assume that $\theta$ is of bounded type, so that the dynamics of $f_{c_0}$ on its Siegel boundary $\partial \mathcal{D}_0$ is characterized by Theorem \ref{siegel boundary rotation}. We are interested in studying the small-scale behavior of this dynamics.

Consider the orbit of the critical point $0$ under $f_{c_0}$:
\begin{displaymath}
\mathcal{O}(0) := \{0, f_{c_0}(0), f_{c_0}^2(0), \ldots\} \subset \partial \mathcal{D}_0.
\end{displaymath}
Denote by $\Delta_n \subset \partial \mathcal{D}_0$ the closed arc containing the critical point $0$ whose end points are $f_{c_0}^{q_n}(0)$ and $f_{c_0}^{q_{n+1}}(0)$, where $q_n$ and $q_{n+1}$ are the $n$th and the $(n+1)$st closest return moments respectively. The arc $\Delta_n$ can be expressed as the union of two closed subarcs $\mathcal{A}_n$ and $\mathcal{A}_{n+1}$, where $\mathcal{A}_n$ has its end points at $0$ and $f_{c_0}^{q_n}(0)$. Observe
\begin{enumerate}[label=(\roman*)]
\item $\mathcal{A}_n \cap \mathcal{A}_{n+1} = \{0\}$;
\item $\mathcal{A}_n \supset \mathcal{A}_{n+2}$;
\item $f_{c_0}^k(\mathcal{A}_n) \cap \Delta_n = \varnothing$ for $0< k < q_{n+1}$ and $f_{c_0}^{q_{n+1}}(\mathcal{A}_n) \subset \Delta_n$; and
\item $f_{c_0}^k(\mathcal{A}_{n+1}) \cap \Delta_n = \varnothing$ for $0< k < q_n$ and $f_{c_0}^{q_n}(\mathcal{A}_{n+1}) \subset \mathcal{A}_n$.
\end{enumerate}
The subarc $\mathcal{A}_n$ is called the $n$th {\it closest return arc}.

The arcs $\Delta_n$ form a nested neighborhood of $0$ in $\partial \mathcal{D}_0$, and by Theorem \ref{siegel boundary rotation}, we see that
\begin{equation}\label{shrinks to point}
\bigcap_{n=0}^\infty \Delta_n = \{0\}.
\end{equation}
Define the $n$th {\it pre-renormalization} $p\mathcal{R}^n(f_{c_0}) : \Delta_{2n} \to \Delta_{2n}$ of $f_{c_0}$ as the {\it first return map} on $\Delta_{2n} = A_{2n} \cup A_{2n+1}$ under iteration of $f_{c_0}$. It is not hard to see that
\begin{displaymath}
p\mathcal{R}^n(f_{c_0}) (x) = \left\{
     \begin{array}{ll}
      f_{c_0}^{q_{2n+1}}(x) & : \text{if } x \in \mathcal{A}_{2n}\\
      f_{c_0}^{q_{2n}}(x) & : \text{if } x \in \mathcal{A}_{2n+1}.
     \end{array}
   \right.
\end{displaymath}
Hence, we can consider $p\mathcal{R}^n(f_{c_0})$ as a pair of maps
\begin{equation}\label{renormalization pair representation}
\hat{\zeta}_n = (\hat{\eta}_n, \hat{\xi}_n) := p\mathcal{R}^n(f_{c_0}) = (f_{c_0}^{q_{2n+1}}|_{\mathcal{A}_{2n}}, f_{c_0}^{q_{2n}}|_{\mathcal{A}_{2n+1}})
\end{equation}
acting on the arc $\Delta_{2n}$. Letting $n=0$, we obtain a pair representation of $f_{c_0}$:
\begin{displaymath}
\hat{\zeta}_{f_{c_0}} : = \hat{\zeta}_0 =p\mathcal{R}^0(f_{c_0}) =  (f_{c_0}^{a_0}|_{\mathcal{A}_0}, f_{c_0}|_{\mathcal{A}_1}).
\end{displaymath}
Intuitively, $p\mathcal{R}^n(f_{c_0})$ captures the dynamics of $f_{c_0}$ on the Siegel boundary $\partial \mathcal{D}_0$ that occurs at the scale of $\Delta_n$.

\begin{figure}[h]
\centering
\includegraphics[scale=0.2]{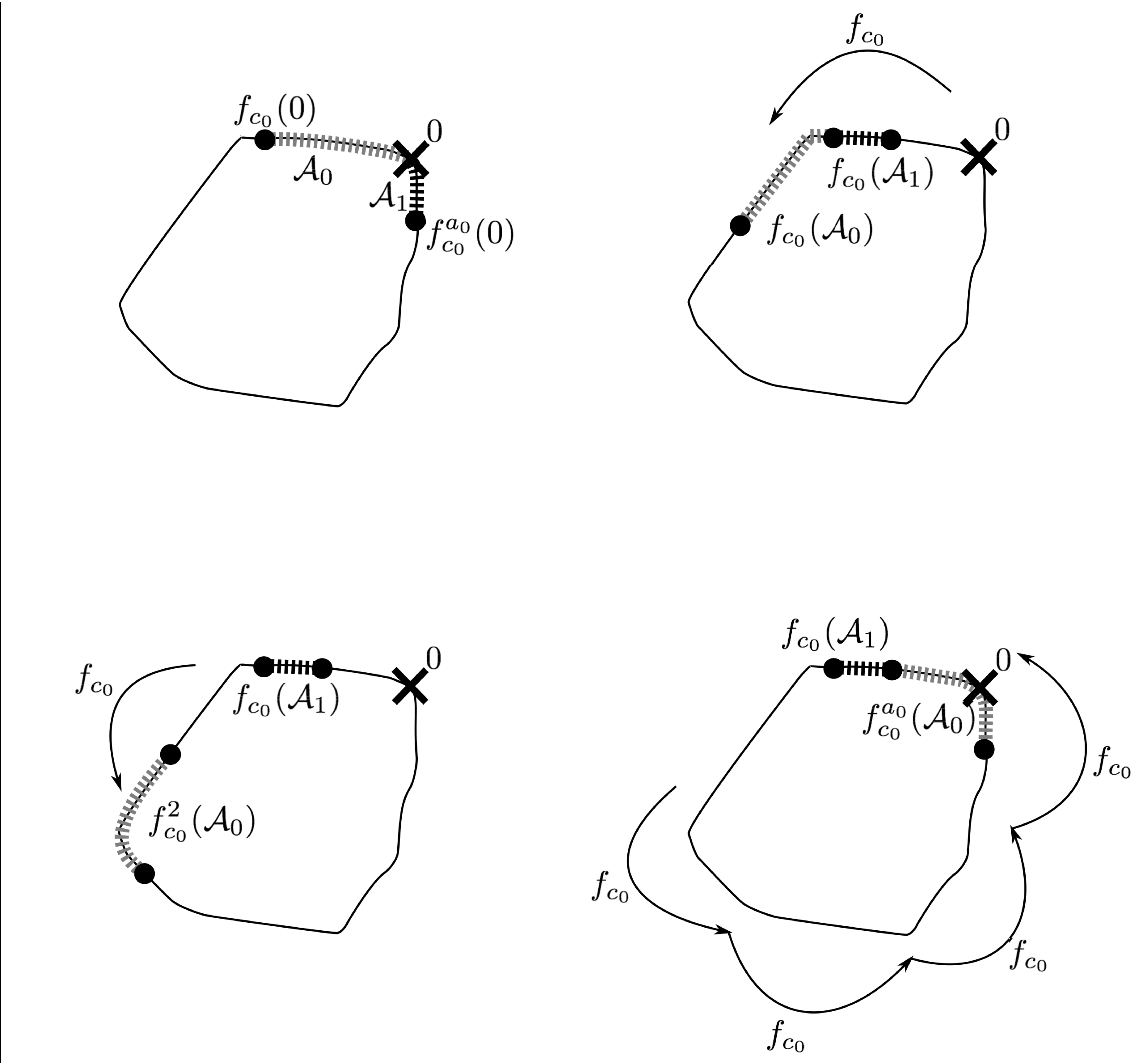}
\caption{The first return map on $\Delta_0 := \mathcal{A}_0 \cup \mathcal{A}_1$ under iteration of $f_{c_0}$ given by the pair $\hat{\zeta}_{f_{c_0}} =  (f_{c_0}^{a_0}|_{\mathcal{A}_0}, f_{c_0}|_{\mathcal{A}_1})$.}
\end{figure}

Note that we can obtain the $(n+1)$st pre-renormalization $\hat{\zeta}_{n+1}$ by taking the first return map on $\Delta_{2(n+1)} \Subset \Delta_{2n}$ under iteration of the $n$th pre-renormalization $\hat{\zeta}_n$:
\begin{displaymath}
\text{i.e.} \hspace{5mm} \hat{\zeta}_{n+1} = p\mathcal{R}(\hat{\zeta}_n).
\end{displaymath}
For the inverse golden mean $\theta_*$, this corresponds to taking the following iterate of $\hat{\zeta}_n$:
\begin{displaymath}
\hat{\zeta}_{n+1} = p\mathcal{R}(\hat{\zeta}_n) = p\mathcal{R}((\hat{\eta}_n, \hat{\xi}_n)) = (\hat{\eta}_n \circ \hat{\xi}_n|_{\mathcal{A}_{2(n+1)}}, \hat{\eta}_n \circ \hat{\xi}_n \circ \hat{\eta}_n|_{\mathcal{A}_{2(n+1)+1}}).
\end{displaymath}
These observations suggest that the sequence of pre-renormalizations of $f_{c_0}$ can be realized as the orbit of $\hat{\zeta}_0 = \hat{\zeta}_{f_{c_0}}$ under the action of some {\it pre-renormalization operator} $p\mathcal{R}$ defined on a space of pairs of maps.

By \eqref{shrinks to point}, we see that $p\mathcal{R}^n(\hat{\zeta}_{f_{c_0}})$ degenerates as $n \to \infty$ to a pair of maps acting on a single point (namely, $0$). To obtain a more meaningful asymptotic behavior, we need to magnify the dynamics of $p\mathcal{R}^n(\hat{\zeta}_{f_{c_0}})$ and bring it to some fixed scale. The simplest way to do this is to conjugate
\begin{displaymath}
p\mathcal{R}^n(\hat{\zeta}_{f_{c_0}}) = \hat{\zeta}_n = (\hat{\eta}_n, \hat{\xi}_n)
\end{displaymath}
by a linear map that sends the critical value $\hat{\xi}_n(0)$ to $1$. The resulting rescaled dynamical system
\begin{displaymath}
\mathcal{R}^n(\zeta_{f_{c_0}}) = \zeta_n =  (\eta_n, \xi_n)
\hspace{5mm} \text{with} \hspace{5mm}
\xi_n(0) = 1
\end{displaymath}
is called the $n$th {\it renormalization} of $f_{c_0}$. If we denote the rescaling operator on pairs by $\Lambda$, we can define the {\it renormalization operator} as
\begin{displaymath}
\mathcal{R} := \Lambda \circ p\mathcal{R}.
\end{displaymath}

Similar to $p\mathcal{R}$, the renormalization operator $\mathcal{R}$ acts on the space of certain pairs of maps. If $\zeta = (\eta, \xi)$ belongs to this space, then it should satisfy the following properties.
\begin{enumerate}[label=\roman*)]
\item The maps $\eta$ and $\xi$ each have a unique simple critical point at $0$.
\item The scale of $\zeta$ is normalized, so that the critical value $\xi(0)$ is at $1$.
\item The maps $\eta$ and $\xi$ extend to holomorphic maps on some neighborhoods $Z$ and $W$ of $0$ in $\mathbb{C}$.
\item Where $\eta$ and $\xi$ are both defined, these maps should commute:
\begin{displaymath}
\eta \circ \xi = \xi \circ \eta.
\end{displaymath}
Observe that commutativity clearly holds for $\zeta = \zeta_n= \mathcal{R}^n(f_{c_0})$, since in this case, $\eta = \eta_n$ and $\xi = \xi_n$ represent different iterates of the same map $f_{c_0}$.
\end{enumerate}

The main goal of this paper is to extend the theory of renormalization to a higher dimensional setting. To this end, consider a quadratic H\'enon map
\begin{displaymath}
H_{c_b,b}(x,y) = (x^2 + c_b - by, x)
\end{displaymath}
that has a semi-Siegel fixed point $\mathbf{p}_b$ with multipliers $\mu = e^{2\pi i \theta}$ and $\nu \in \mathbb{D} \setminus \{0\}$. Such H\'enon maps are parameterized by their Jacobian:
\begin{displaymath}
b = \nu / \mu \equiv \text{Jac} \, H_{c_b,b}.
\end{displaymath}
As $b \to 0$, the semi-Siegel H\'enon map $H_{c_b,b}$ degenerates to the following two-dimensional embedding of the quadratic Siegel polynomial $f_{c_0}$:
\begin{displaymath}
(x,y) \mapsto (f_{c_0}(x), x).
\end{displaymath}
Hence, for $|b| <<1$, the dynamics of $H_{c_b, b}$ can be considered as a small perturbation of the dynamics of $f_{c_0}$.

Let $\mathcal{D}_b$ be the two-dimensional Siegel disk of $H_{c_b, b}$. {\it A priori}, we do not have an analog of Theorem \ref{siegel boundary rotation} that characterizes the dynamics of $H_{c_b, b}$ on $\partial \mathcal{D}_b$. However, we can still define the $n$th {\it pre-renormalization} $p\mathbf{R}^n(H_{c_b, b})$ of $H_{c_b, b}$ by taking the same iterates as in \eqref{renormalization pair representation}:
\begin{equation}\label{2d renormalization pair representation}
p\mathbf{R}^n(H_{c_b, b}) = \hat \Sigma_n = (\hat A_n, \hat B_n) := (H_{c_b, b}^{q_{2n+1}}|_{\Omega_n}, H_{c_b, b}^{q_{2n}}|_{\Gamma_n}).
\end{equation}
In \eqref{2d renormalization pair representation}, the sets $\Omega_n$ and $\Gamma_n$ are chosen to be some suitable domains in $\mathbb{C}^2$ which intersect $\partial \mathcal{D}_b$. By letting $n=0$, we obtain a pair representation of $H_{c_b, b}$:
\begin{displaymath}
\hat \Sigma_{H_{c_b, b}} := \hat \Sigma_0 = p\mathbf{R}^0(H_{c_b, b}) = (H_{c_b, b}^{a_0}|_{\Omega_0}, H_{c_b, b}|_{\Gamma_0}).
\end{displaymath}

Analogously to the one-dimensional case, the sequence of pre-renormalizations of $H_{c_b, b}$ can be realized as the orbit of $\hat \Sigma_{H_{c_b, b}}$ under the action of some {\it pre-renormalization operator} $p\mathbf{R}$ defined on a space of pairs of two-dimensional maps. To transform $p\mathbf{R}$ into a proper {\it renormalization operator} $\mathbf{R}$, we need to compose $p\mathbf{R}$ with some suitable {\it rescaling operator} $\mathbf{\Lambda}$. However, this turns out to be more a intricate problem in two-dimensions than in the one-dimensional case. To ensure a tractable asymptotic behavior under iterations of $\mathbf{R}$, it is not only important to fix the scale of the dynamical systems, but we must also bring them back to H\'enon-like form after each renormalization. To achieve this, we incorporate a non-linear change of coordinates to the definition of $\mathbf{\Lambda}$. Further details are provided in Section \ref{sec:renormalization}.

Suppose that the renormalizations of $H_{c_b, b}$ are given by
\begin{displaymath}
\mathbf{R}^n(\Sigma_{H_{c_b, b}}) = \Sigma_n = (A_n, B_n),
\end{displaymath}
where $A_n$ and $B_n$ are defined on some fixed neighborhoods $\Omega$ and $\Gamma$ of $(0,0)$ in $\mathbb{C}^2$. Recall that $A_n$ and $B_n$ represent rescalings of the $q_{2n+1}$ and $q_{2n}$ iterates of $H_{c_b, b}$ respectively. If $H_{c_b, b}$ is sufficiently dissipative, so that $|b| < \epsilon$ for some $\epsilon <1$, then by the chain rule, the Jacobians of $A_n$ and $B_n$ are on the order of $\epsilon^{q_{2n+1}}$ and $\epsilon^{q_{2n}}$ respectively. Hence, if the renormalization sequence $\{\Sigma_n\}_{n=0}^\infty$ converges to some limit $\Sigma_* = (A_*, B_*)$, then we have
\begin{displaymath}
\text{Jac} \, A_* = \lim_{n\to \infty} O(\epsilon^{q_{2n+1}})= 0
\hspace{5mm} \text{and} \hspace{5mm}
\text{Jac} \, B_* = \lim_{n\to \infty} O(\epsilon^{q_{2n}}) =0.
\end{displaymath}
Thus, we see that the limit $\Sigma_*$ of the renormalizations of $H_{c_b, b}$ must be a degenerate one-dimensional system.

\section{Renormalization of Almost Commuting Pairs} \label{sec:renormalization}

In this section, we formalize the ideas discussed in Section \ref{sec:motivation}. While previously, we considered any rotation number $\theta$ of bounded type, we will henceforth restrict to the case of the inverse golden-mean:
\begin{displaymath}
\theta_* = \frac{\sqrt{5}-1}{2} = [1, 1, \ldots].
\end{displaymath}

\subsection*{One-dimensional renormalization}
For a domain $Z \subset \mathbb{C}$, we denote by $\mathcal{A}(Z)$ the Banach space of bounded analytic functions $f : Z \to \mathbb{C}$, equipped with the norm
\begin{displaymath}
\| f\|= \sup_{x \in Z}|f(x)|.
\end{displaymath}

Denote by $\mathcal{A}(Z, W)$ the Banach space of bounded pairs of analytic functions $\zeta=(f, g)$ from domains $Z \subset \mathbb{C}$ and $W \subset \mathbb{C}$ respectively to $\mathbb{C}$, equipped with the norm
\begin{displaymath}
\|\zeta\|= \frac{1}{2} \left(   \|f\|+  \|g\| \right).
\end{displaymath}
Henceforth, we assume that the domains $Z$ and $W$ contain $0$.

For a pair $\zeta=(f,g)$, define the \emph{rescaling map} as
\begin{displaymath}
\Lambda(\zeta) := (s_\zeta^{-1} \circ f \circ s_\zeta, s_\zeta^{-1} \circ g \circ s_\zeta),
\end{displaymath}
where
\begin{displaymath}
s_\zeta(x) := \lambda_\zeta x
\hspace{5mm} \text{and} \hspace{5mm}
\lambda_\zeta := g(0).
\end{displaymath}

\begin{defn}\label{1d crit pair}
We say that $\zeta= (\eta, \xi) \in \mathcal{A}(Z,W)$ is a {\it critical pair} if
\begin{enumerate}[label=(\roman{*})]
\item $\eta$ and $\xi$ have a simple unique critical point at $0$, and
\item $\xi(0)=1$.
\end{enumerate}
The space of critical pairs in $\mathcal{A}(Z,W)$ is denoted by $\mathcal{C}(Z,W)$.
\end{defn}

\begin{figure}[h]
\centering
\includegraphics[scale=0.36]{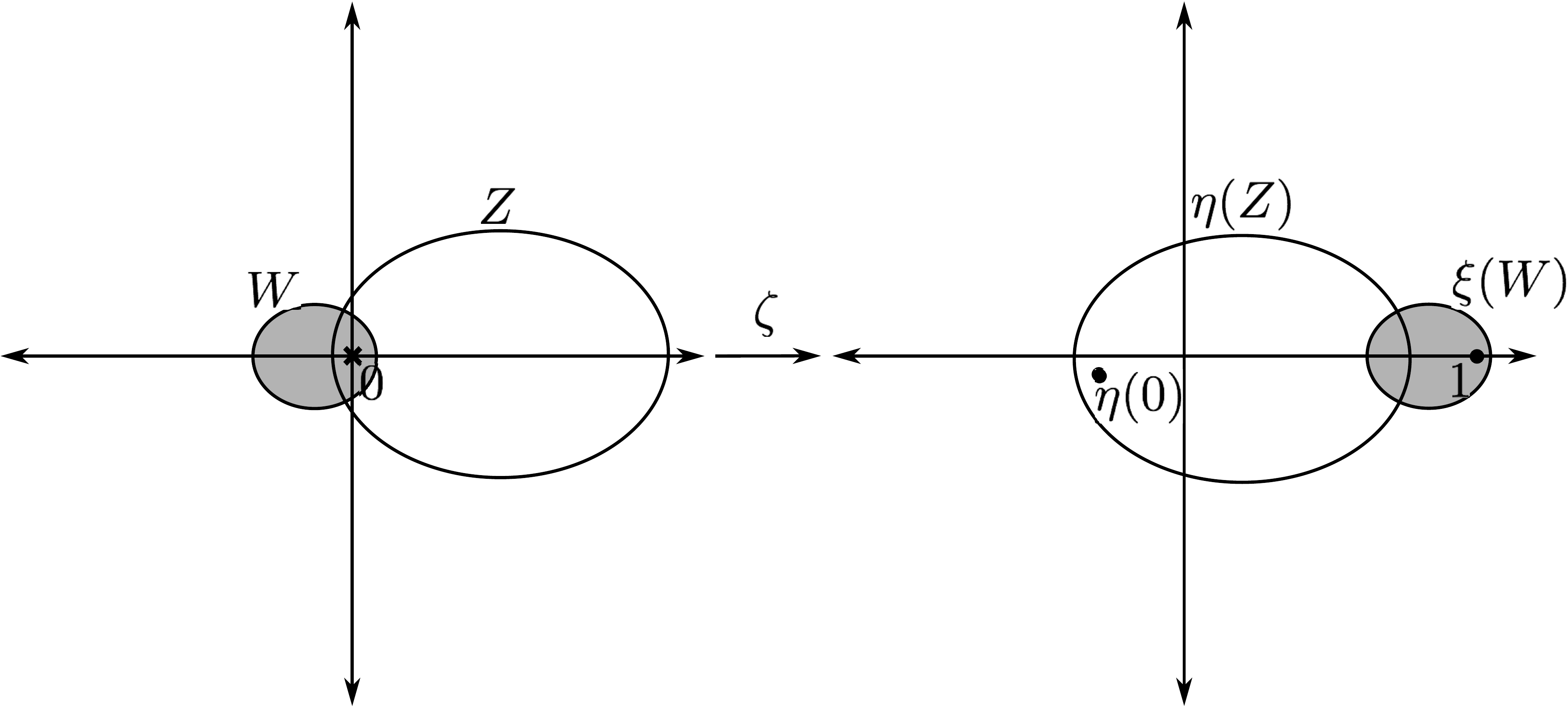}
\caption{A critical pair $\zeta=(\eta, \xi) \in \mathcal{C}(Z,W)$.}
\label{fig:critpair}
\end{figure}

\begin{defn}\label{1d commute pair}
We say that $\zeta= (\eta, \xi) \in \mathcal{A}(Z,W)$ is a {\it commuting pair} if
\begin{displaymath}
\eta \circ \xi = \xi \circ \eta.
\end{displaymath}
\end{defn}

It turns out that requiring strict commutativity is too restrictive in the category of analytic functions. Hence, we work with the following less restrictive condition.

\begin{defn}\label{1d almost commute pair}
We say that $\zeta= (\eta, \xi) \in \mathcal{C}(Z,W)$ is an {\it almost commuting pair} (cf. \cite{Bur,Stir}) if
\begin{displaymath}
\frac{d^i (\eta \circ \xi - \xi \circ \eta)}{dx^i}(0) = 0
\hspace{5mm} \text{for} \hspace{5mm}
i = 0, 2.
\end{displaymath}
The space of almost commuting pairs in $\mathcal{C}(Z,W)$ is denoted by $\mathcal{B}(Z, W)$.
\end{defn}

Note that if $\zeta = (\eta, \xi)$ is a critical pair, then the first-order commuting relation is automatically satisfied:
\begin{displaymath}
\frac{d (\eta \circ \xi - \xi \circ \eta)}{dx}(0) =  \eta'(1)\xi'(0) - \xi'(\eta(0))\eta'(0) = 0.
\end{displaymath}

\begin{prop}[cf. \cite{GaYa2}]
The spaces $\mathcal{C}(Z,W)$ and $\mathcal{B}(Z,W)$ have the structure of an immersed Banach submanifold of $\mathcal{A}(Z,W)$ of codimension $3$ and $5$ respectively.
\end{prop}

\begin{defn} \label{def:1d pre-renormalization}
Let $\zeta = (\eta, \xi) \in \mathcal{B}(Z,W)$. The {\it pre-renormalization} of $\zeta$ is defined as:
\begin{displaymath}
p\mathcal{R}(\zeta) := (\eta \circ \xi \circ \eta, \eta \circ \xi).
\end{displaymath}
The {\it renormalization} of $\zeta$ is defined as:
\begin{displaymath}
\mathcal{R}(\zeta) := \Lambda(p\mathcal{R}(\zeta)).
\end{displaymath}
We say that $\zeta$ is {\it renormalizable} if $\mathcal{R}(\zeta) \in \mathcal{B}(Z,W)$.
\end{defn}

\begin{figure}[h]
\centering
\includegraphics[scale=0.2]{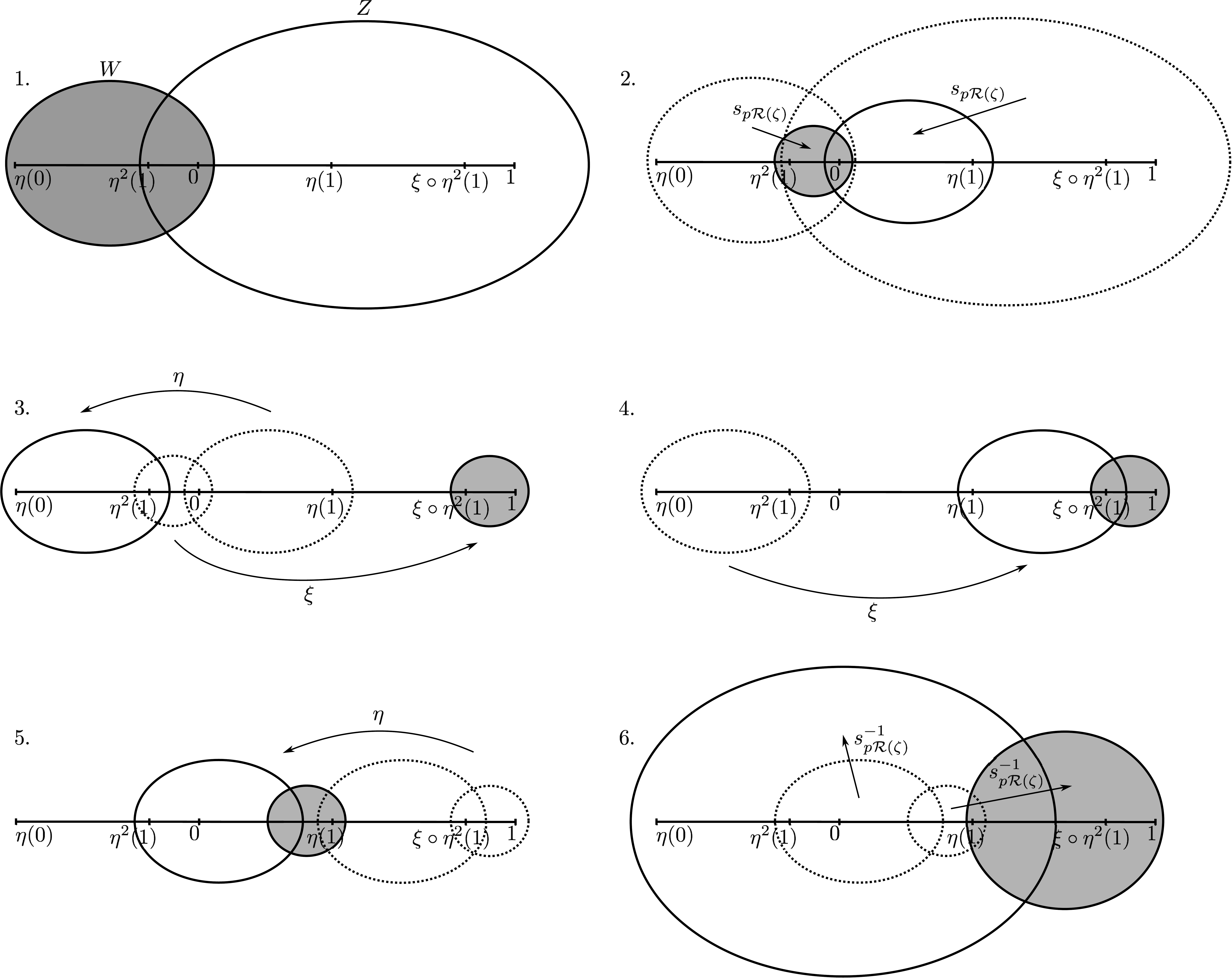}
\caption{The 1D-renormalization $\mathcal{R}(\zeta) := \Lambda(p\mathcal{R}(\zeta)) :=  \Lambda((\eta \circ \xi \circ \eta, \eta \circ \xi))$.}
\label{fig:1drenormalization}
\end{figure}

The following is shown in \cite{GaYa2}:

\begin{thm}[1D-Renormalization Hyperbolicity] \label{1d renormalization}
There exist topological disks $Z' \Supset Z$ and $W' \Supset W$, and a commuting pair $\zeta_* = (\eta_*, \xi_*) \in \mathcal{B}(Z,W)$ such that the following holds:
\begin{enumerate}[label=(\roman{*})]
\item There exists a neighborhood $\mathcal{N}$ of $\zeta_*$ in the submanifold $\mathcal{B}(Z,W)$ such that
\begin{displaymath}
\mathcal{R} : \mathcal{N} \to \mathcal{B}(Z',W')
\end{displaymath}
is an analytic operator.
\item The pair $\zeta_*$ is the unique fixed point of $\mathcal{R}$ in $\mathcal{N}$. In particular, we have
\begin{displaymath}
\lambda_*^{-1} \eta_* \circ \xi_* \circ \eta_* (\lambda_* x) = \eta_*(x)
\hspace{5mm} \text{and} \hspace{5mm}
\lambda_*^{-1} \eta_* \circ \xi_* (\lambda_* x)=\xi_*(x),
\end{displaymath}
where
\begin{displaymath}
\lambda_* := \eta_* \circ \xi_* (0)
\end{displaymath}
is a universal scaling factor.
\item The differential $D_{\zeta_*}\mathcal{R}$ is a compact linear operator. Moreover, $D_{\zeta_*}\mathcal{R}$ has a single, simple eigenvalue with modulus greater than $1$. The rest of its spectrum lies inside the open unit disk $\mathbb{D}$ (and hence is compactly contained in $\mathbb{D}$ by the spectral theory of compact operators).
\end{enumerate}
\end{thm}

Let
\begin{displaymath}
f_*(z) = z^2 +c_*
\end{displaymath}
be the quadratic polynomial with a Siegel fixed point of multiplier $\mu_* = e^{2\pi i \theta_*}$, where $\theta_* = (\sqrt{5}-1)/2$ is the inverse golden-mean rotation number. For $c$ sufficiently close to $c_*$, we can identify the quadratic polynomial $f_c$ as a pair in $\mathcal{B}(Z, W)$ as follows:
\begin{equation}\label{eq:quadpair}
\zeta_{f_c} := \Lambda(f_c^2|_{Z_c}, f_c|_{W_c}),
\end{equation}
where
\begin{displaymath}
Z_c := s_{f_c}(Z) = f_c(0)\cdot Z
\hspace{5mm} \text{and} \hspace{5mm}
W_c := s_{f_c}(W) = f_c(0) \cdot W
\end{displaymath}

The following is shown in \cite{GaRYa}:

\begin{thm} \label{quadratic}
The one-parameter family $\{\zeta_{f_c}\}_c$ intersects the stable manifold $W^s(\zeta_*) \subset \mathcal{B}(Z, W)$ of the fixed point $\zeta_*$ for the 1D-renormalization operator $\mathcal{R}$. Moreover, this intersection is transversal, and occurs at $\zeta_{f_*}$.
\end{thm}

\begin{figure}[h]
\centering
\includegraphics[scale=0.4]{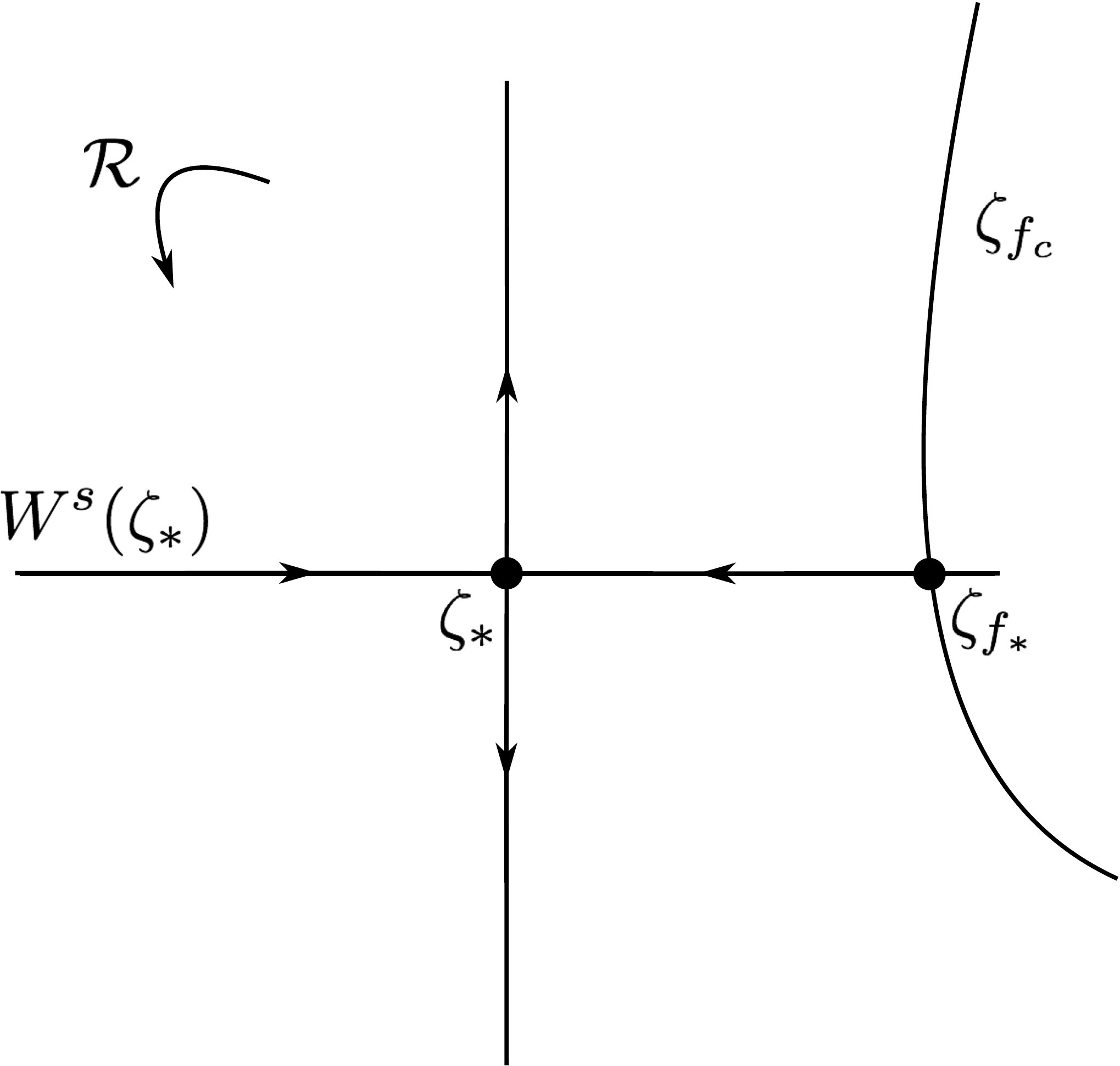}
\caption{The stable manifold $W^s(\zeta_*)$ of the fixed point $\zeta_*$ for the 1D-renormalization operator $\mathcal{R}$. The family of quadratic polynomials $\zeta_{f_c}$ intersect $W^s(\zeta_*)$ transversely at the golden-mean Siegel quadratic polynomial $\zeta_{f_*}$.}
\label{fig:1dhyperbolicity}
\end{figure}

\subsection*{Two-dimensional renormalization}

For a domain $\Omega \subset \mathbb{C}^2$, we denote by $\mathcal{A}_2(\Omega)$ the Banach space of bounded analytic functions $F : \Omega \to \mathbb{C}^2$, equipped with the norm
\begin{displaymath}
\| F\|= \sup_{(x,y) \in \Omega}|F(x,y)|.
\end{displaymath}
Define
\begin{displaymath}
\| F\|_y := \sup_{(x,y) \in \Omega}|\partial_y F(x,y)|.
\end{displaymath}

Denote by $\mathcal{A}_2(\Omega, \Gamma)$ the Banach space of bounded pairs of analytic functions $\Sigma=(F,G)$  from domains $\Omega \subset \mathbb{C}^2$ and $\Gamma \subset \mathbb{C}^2$ respectively to $\mathbb{C}^2$, equipped with the norm
\begin{displaymath}
\| \Sigma\|= \frac{1}{2} \left(\|F\|+  \|G\|\right).
\end{displaymath}
Define
\begin{displaymath}
\| \Sigma\|_y :=  \frac{1}{2} \left(\|F\|_y +  \|G\|_y \right).
\end{displaymath}

Henceforth, we assume that
\begin{displaymath}
\Omega=Z  \times U \hspace{5mm} \text{and} \hspace{5mm} \Gamma=W \times V,
\end{displaymath}
where $Z$, $U$, $W$ and $V$ are domains of $\mathbb{C}$ containing $0$. For a function
\begin{displaymath}
F(x,y):=
\begin{bmatrix}
f_1(x,y)\\
f_2(x,y)
\end{bmatrix}
\end{displaymath}
from $\Omega$ or $\Gamma$ to $\mathbb{C}^2$, we define the \emph{projection map} as
\begin{displaymath}
 \pi_1F(x):=f_1(x,0).
\end{displaymath}
For a pair $\Sigma = (F, G)$, we define the \emph{projection map} as
\begin{displaymath}
\pi_1 \Sigma := (\pi_1 F, \pi_1 G),
\end{displaymath}
and the \emph{rescaling map} as
\begin{displaymath}
\Lambda(\Sigma) := (s_\Sigma^{-1} \circ F \circ s_\Sigma, s_\Sigma^{-1} \circ G \circ s_\Sigma),
\end{displaymath}
where
\begin{displaymath}
s_\Sigma(x,y) := (\lambda_\Sigma x, \lambda_\Sigma y)
\hspace{5mm} \text{and} \hspace{5mm}
\lambda_\Sigma :=  \pi_1 G(0).
\end{displaymath}

The following definitions are analogs of Definition \ref{1d crit pair}, \ref{1d commute pair} and \ref{1d almost commute pair}.

\begin{defn}
For $\epsilon \geq 0$, we say that $\Sigma=(A,B) \in \mathcal{A}_2(\Omega, \Gamma)$ is an {\it $\epsilon$-critical pair} if
\begin{enumerate}[label=(\roman{*})]
\item $\pi_1 A$ and $ \pi_1 B$ have a simple unique critical point which is contained in a $\epsilon$-neighborhood of $0$, and
\item $\pi_1 B(0)=1$.
\end{enumerate} The space of $\epsilon$-critical pairs in $\mathcal{A}_2(\Omega, \Gamma)$ is denoted by $\mathcal{C}_2(\Omega,\Gamma, \epsilon)$.
\end{defn}

\begin{defn}
We say that $\Sigma=(A,B) \in \mathcal{A}_2(\Omega, \Gamma)$ is a \emph{commuting pair} if
\begin{displaymath}
A \circ B = B \circ A.
\end{displaymath}
\end{defn}

\begin{defn}
We say that $\Sigma= (A, B) \in \mathcal{C}_2(\Omega,\Gamma, \epsilon)$ is an {\it $\epsilon$-almost commuting pair} if
\begin{displaymath}
\left | \frac{d^i  \pi_1[A, B]}{dx^i}(0) \right | := \left | \frac{d^i  \pi_1(A \circ B - B \circ A)}{dx^i}(0) \right | \leq \epsilon
\hspace{5mm} \text{for} \hspace{5mm}
i = 0, 2.
\end{displaymath}
The space of $\epsilon$-almost commuting pairs in $\mathcal{C}_2(\Omega,\Gamma, \epsilon)$ is denoted by $\mathcal{B}_2(\Omega,\Gamma, \epsilon)$.
\end{defn}

\begin{notn}\label{1d twice ren pairs}
We denote by $\mathcal{D}(Z, W) \subset \mathcal{B}(Z, W)$ the set of 1D-renormalizable pairs. That is, a pair $\zeta_0 = (\eta_0, \xi_0)$ is in $\mathcal{D}(Z,W)$ if
\begin{displaymath}
\mathcal{R}(\zeta_0) = \Lambda((\eta_0 \circ \xi_0 \circ \eta_0, \eta_0 \circ \xi_0))
\end{displaymath}
is a well-defined element of $\mathcal{B}(Z, W)$.
\end{notn}

We define an embedding $\iota$ of $\mathcal{D}(Z,W)$ into $\mathcal{B}_2(\Omega, \Gamma,0)$ as follows. For a pair $\zeta = (\eta, \xi) \in \mathcal{D}(Z,W)$, we let
\begin{displaymath}
\iota(\zeta) := \Lambda((A_\zeta, B_\zeta)),
\end{displaymath}
where
\begin{equation}\label{eq:embedding}
A_\zeta(x,y):= \begin{bmatrix}
\eta \circ \xi \circ \eta (x) \\
\eta(x)
\end{bmatrix}
\hspace{5mm} \text{and} \hspace{5mm}
B_\zeta(x,y):= \begin{bmatrix}
\eta \circ \xi (x) \\
x
\end{bmatrix}.
\end{equation}
Observe
\begin{equation}\label{renorm embedding}
\pi_1 \iota(\zeta) = \mathcal{R}(\zeta).
\end{equation}

\begin{lem} \label{renormalization nbh}
Let $\tilde Z \Subset Z$ and $\tilde W \Subset W$ be domains in $\mathbb{C}$. For any $\zeta_0 \in \mathcal{D}(Z,W)$, there exists a neighborhood $\mathcal{N}(\zeta_0) \subset \mathcal{A}(Z, W)$ of $\zeta_0$ such that if $\zeta = (\eta, \xi) \in \mathcal{N}(\zeta_0)$, then the pair
\begin{displaymath}
\mathcal{R}(\zeta) := \Lambda((\eta \circ \xi \circ \eta, \eta \circ \xi))
\end{displaymath}
is a well-defined element of $\mathcal{A}(\tilde Z, \tilde W)$.
\end{lem}

For $\epsilon \geq 0$, we denote by $\mathcal{D}_2(\Omega, \Gamma, \epsilon) \subset \mathcal{B}_2(\Omega, \Gamma, \epsilon)$ the set of pairs $\Sigma=(A,B)$ of the form
\begin{displaymath}
A(x,y) = \begin{bmatrix}
a(x,y) \\
h(x,y)
\end{bmatrix}
\hspace{5mm} \text{and} \hspace{5mm}
B(x,y) = \begin{bmatrix}
b(x,y) \\
x
\end{bmatrix}
\end{displaymath}
such that the following conditions are satisfied.
\begin{enumerate}[label=(\roman{*})]
\item For
\begin{displaymath}
\zeta := \pi_1\Sigma,
\end{displaymath}
the pair $\mathcal{R}(\zeta)$ defined in Lemma \ref{renormalization nbh} is a well-defined element of $\mathcal{A}(\tilde Z, \tilde W)$, where
\begin{displaymath}
\tilde Z := (1-\epsilon)Z
\hspace{5mm} \text{and} \hspace{5mm}
\tilde W : = (1-\epsilon)W.
\end{displaymath}
\item We have $\|\Sigma\|_y\leq \epsilon$ and $\|h\| < 2$.
\end{enumerate}

\begin{lem}
The set $\mathcal{D}_2(\Omega, \Gamma, \epsilon)$ have the structure of an immersed Banach submanifold of $\mathcal{A}_2(\Omega,\Gamma)$.
\end{lem}

We define the renormalization of $\Sigma \in \mathcal{D}_2(\Omega, \Gamma, \epsilon)$ in several steps. First, we define the {\it pre-renormalization} of $\Sigma$ as
\begin{equation}\label{eq:2d pre-renormalization}
p\mathbf{R}(\Sigma)=(A_1, B_1) := (B \circ A^2, B \circ A)
\end{equation}
Next, we denote
\begin{displaymath}
a_y(x):= a(x,y),
\end{displaymath}
and consider the following non-linear changes of coordinates:
\begin{equation} \label{eq:nonlinear change of coord}
H(x,y):=
\begin{bmatrix}
a^{-1}_y(x) \\
y
\end{bmatrix}.
\end{equation}
Define
\begin{displaymath}
p\tilde{\mathbf{R}}(\Sigma) = (A_2, B_2) := (H^{-1} \circ A_1 \circ H, H^{-1} \circ B_1 \circ H).
\end{displaymath}

\begin{figure}[h]
\centering
\includegraphics[scale=0.22]{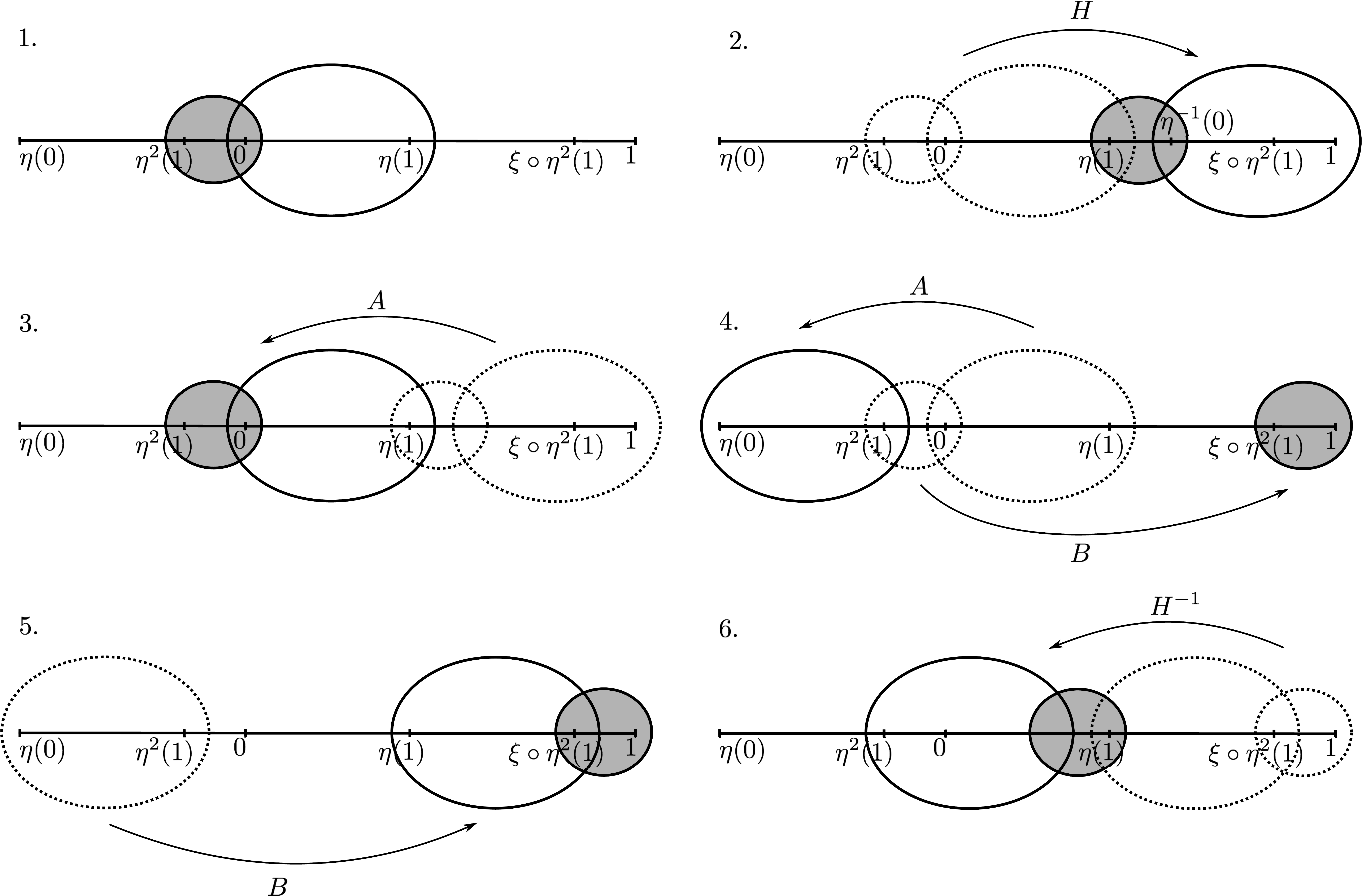}
\caption{The 2D-pre-renormalization $p\tilde{\mathbf{R}}(\Sigma) := (H^{-1} \circ B \circ A^2 \circ H, H^{-1} \circ B \circ A \circ H)$.}
\label{fig:pR}
\end{figure}

Let
\begin{displaymath}
\zeta=(\eta, \xi) := \pi_1 \Sigma.
\end{displaymath}
It is not hard to check that we have the following estimates:
\begin{equation}\label{eq:renormalization estimates}
\|\Lambda(p\tilde{\mathbf{R}}(\Sigma))-\iota(\zeta)\| = O(\epsilon)
\hspace{5mm} \text{and} \hspace{5mm}
\|p\tilde{\mathbf{R}}(\Sigma)\|_y = O(\epsilon^2),
\end{equation}
where $\iota : \mathcal{D}(Z, W) \to \mathcal{B}_2(\Omega, \Gamma, 0)$ is the embedding given in \eqref{eq:embedding}.

By \eqref{eq:renormalization estimates} and the argument principle, it follows that if $\epsilon$ is sufficiently small, then the function $ \pi_1 B_2 \circ A_2$ has a simple unique critical point $c_a$ near $0$. Set
\begin{equation}\label{eq:crit trans}
T_a(x,y):= (x + c_a, y),
\end{equation}
and define
\begin{displaymath}
\tilde{\mathbf{R}}(\Sigma) = (A_3, B_3) := \Lambda((T_a^{-1} \circ A_2 \circ T_a, T_a^{-1} \circ B_2 \circ T_a)).
\end{displaymath}
Observe that we have
\begin{displaymath}
0 =  \pi_1(B_3 \circ A_3)'(0) = \|B_3\|(\pi_1 A_3)'(0) + O(\epsilon^2).
\end{displaymath}
If $\Sigma = (A,B)$ is a commuting pair, then $A_3$ and $B_3$ would also commute. In this case, we would have
\begin{displaymath}
0 =  \pi_1(B_3 \circ A_3)'(0) = \pi_1(A_3 \circ B_3)'(0) = \|A_3\|(\pi_1 B_3)'(0) + O(\epsilon^2).
\end{displaymath}
Thus, we see that the operator $\tilde{\mathbf{R}}$ maps commuting pairs in $\mathcal{D}_2(\Omega, \Gamma, \epsilon)$ into $\mathcal{B}_2(\Omega, \Gamma, \delta)$, where $\delta = O(\epsilon^2)$. To finish the definition of the renormalization operator, we need to compose $\tilde{\mathbf{R}}$ with an appropriate projection operator, so that the image of non-commuting pairs map into $\mathcal{B}_2(\Omega, \Gamma, \delta)$ as well.

\begin{lem}
There exists an analytic projection operator $\Pi$ defined on $\tilde{\mathbf{R}}(\mathcal{D}_2(\Omega, \Gamma, \epsilon))$ such that for any $\Sigma = (A,B) \in \mathcal{D}_2(\Omega, \Gamma, \epsilon)$, the following statements hold:
\begin{enumerate}[label=(\roman{*})]
\item $\Pi \circ \tilde{\mathbf{R}}(\Sigma) \in \mathcal{B}_2(\Omega, \Gamma, \delta)$, where $\delta = O(\epsilon^2)$,
\item $\|\Pi \circ \tilde{\mathbf{R}}(\Sigma) - \iota(\zeta)\| = O(\epsilon)$, where $\zeta := (\pi_1 A, \pi_1 B)$ and $\iota$ is the embedding given in \eqref{eq:embedding}, and
\item if $\Sigma$ is a commuting pair, then $\Pi \circ \tilde{\mathbf{R}}(\Sigma) = \tilde{\mathbf{R}}(\Sigma)$.
\end{enumerate}
\end{lem}

\begin{proof}
By \eqref{eq:renormalization estimates} and the argument principle, it follows that if $\epsilon$ is sufficiently small, then the function $ \pi_1 A_2 \circ B_2$ has a simple unique critical point $c_b$ near $0$. Set
\begin{displaymath}
T_b(x,y):= (x + c_b, y),
\end{displaymath}
and define
\begin{displaymath}
\Pi_{\text{crit}} \circ \tilde{\mathbf{R}}(\Sigma) = (A_4, B_4) := \Lambda((T_b^{-1} \circ T_a^{-1} \circ A_2 \circ T_a, T_a^{-1} \circ B_2 \circ T_a \circ T_b)).
\end{displaymath}
Observe that we have
\begin{displaymath}
0 = \pi_1(B_4 \circ A_4)'(0) = \|B_4\|(\pi_1 A_4)'(0) + O(\epsilon^2),
\end{displaymath}
and similarly
\begin{displaymath}
0 = \pi_1(A_4 \circ B_4)'(0) = \|A_4\|(\pi_1 B_4)'(0) + O(\epsilon^2).
\end{displaymath}
Hence, we see that $\Pi_{\text{crit}} \circ \tilde{\mathbf{R}}(\Sigma)$ is in $\mathcal{C}_2(\Omega, \Gamma, \tilde{\delta})$ for some $\tilde{\delta} = O(\epsilon^2)$.

Write
\begin{displaymath}
A_4(x,y) = \begin{bmatrix}
a_4(x,y) \\
h_4(x,y)
\end{bmatrix}
\hspace{5mm} \text{and} \hspace{5mm}
B_4(x,y) = \begin{bmatrix}
b_4(x,y) \\
x
\end{bmatrix}.
\end{displaymath}
To project $\Pi_{\text{crit}} \circ \tilde{\mathbf{R}}(\Sigma)$ into $\mathcal{B}_2(\Omega, \Gamma, \delta)$ for some $\tilde{\delta} = O(\epsilon^2)$, we make the following modifications:
\begin{displaymath}
\Pi \circ \tilde{\mathbf{R}}(\Sigma) :=\Pi_{\text{ac}} \circ \Pi_{\text{crit}} \circ \tilde{\mathbf{R}}(\Sigma)=(A_5, B_5),
\end{displaymath}
where
\begin{displaymath}
A_5(x,y) = \begin{bmatrix}
a_4(x,y) \\
h_4(x,y)
\end{bmatrix}
\hspace{5mm} \text{and} \hspace{5mm}
B_5(x,y) = \begin{bmatrix}
b_4(x,y)+cx^3+dx^4 \\
x
\end{bmatrix}.
\end{displaymath}
Observe that $\Pi \circ \tilde{\mathbf{R}}(\Sigma)$ still belongs to $\mathcal{C}_2(\Omega, \Gamma, \tilde{\delta})$. To determine the constants $c$ and $d$, we compute:
\begin{align*}
0=\pi_1[A_5, B_5](0)=\pi_1[A_4, B_4](0)-ca_4(0,0)^3-da_4(0,0)^4,
\end{align*}
and
\begin{align*}
0&=\frac{d^2  \pi_1[A_5, B_5]}{dx^2}(0)\\
&=\frac{d^2  \pi_1[A_4, B_4]}{dx^2}(0)-3ca_4(0,0)^2\partial_x^2 a_4(0,0)-4da_4(0,0)^3\partial_x^2 a_4(0,0)+O(\epsilon^2).
\end{align*}
Since $a_4(0,0)$ and $\partial_x^2 a_4(0,0)$ are both uniformly bounded away from $0$, we can solve the above equations for $c$ and $d$.

By assumption, $\Sigma \in \mathcal{D}_2(\Omega, \Gamma, \epsilon) \subset \mathcal{B}_2(\Omega, \Gamma, \epsilon)$. This means that we have
\begin{displaymath}
\pi_1[A_4, B_4](0)=O(\epsilon)
\hspace{5mm}\text{and}\hspace{5mm}
\frac{d^2  \pi_1[A_4, B_4]}{dx^2}(0) =O(\epsilon).
\end{displaymath}
The lemma now follows from observing that $c$ and $d$ must be of the same order as these commutators, and in fact, must be equal to $0$ if $\Sigma = (A,B)$ is a commuting pair.
\end{proof}

We can now define the {\it renormalization} of $\Sigma = (A,B) \in \mathcal{D}_2(\Omega, \Gamma, \epsilon)$ as:
\begin{displaymath}
\mathbf{R}(\Sigma) := \Pi \circ \tilde{\mathbf{R}}(\Sigma).
\end{displaymath}

The following theorem provides a summary of our construction.

\begin{thm}[2D-Renormalization Hyperbolicity] \label{renormalization hyperbolicity}
The 2D-renormalization operator
\begin{displaymath}
\mathbf{R} : \mathcal{D}_2(\Omega, \Gamma, \epsilon) \to \mathcal{B}_2(\Omega, \Gamma, \infty)
\end{displaymath}
is well-defined and analytic. Moreover, let $\zeta_* \in \mathcal{D}(Z, W)$ be the fixed point of the 1D-renormalization given in Theorem \ref{1d renormalization}, and let $\iota(\zeta_*)=(A_*, B_*)$ be its embedding into $\mathcal{D}_2(\Omega, \Gamma, 0)$. For $\epsilon > 0$, there exists a neighborhood $\mathcal{N} \Subset \mathcal{D}_2(\Omega, \Gamma, \epsilon)$ of $\iota(\zeta_*)$, and a constant $C < 1/\epsilon$ depending on $\mathcal{N}$ such that the following statements hold.

\begin{enumerate}[label=(\roman{*})]
\item If $\Sigma = (A,B) \in \mathcal{N}$ and $\zeta := (\pi_1 A, \pi_1 B)$, then $\mathbf{R}(\Sigma) \in \mathcal{B}_2(\Omega, \Gamma, C\epsilon^2)$, and
\begin{displaymath}
\|\mathbf{R}(\Sigma) - \iota(\zeta)\| <C \epsilon.
\end{displaymath}
Hence, by \eqref{renorm embedding}, we have
\begin{displaymath}
\|\pi_1 \mathbf{R}(\Sigma) - \mathcal{R}(\zeta)\| < C \epsilon.
\end{displaymath}

\item The pair $\iota(\zeta_*)$ is the unique fixed point of $\mathbf{R}$ in $\mathcal{N}$.

\item The differential $D_{\iota(\zeta_*)}\mathbf{R}$ is a compact linear operator whose spectrum coincides with that of  $D_{\zeta_*}\mathcal{R}$. More precisely, let $N \Subset \mathcal{D}(Z,W)$ be a small neighborhood of $\zeta_*$. Then in the spectral decomposition of $D_{\iota(\zeta_*)}\mathbf{R}$, the complement to the tangent space of $\iota(N)$ corresponds to the zero eigenvalue.
\end{enumerate}
\end{thm}

Let $H_{\mu_*,\nu}$ be the H{\'e}non map with a semi-Siegel fixed point $\mathbf{p}$ of multipliers $\mu_* = e^{2\pi i \theta_*}$ and $\nu$, where $\theta_* = (\sqrt{5}-1)/2$ is the inverse golden mean rotation number, and $|\nu|<\epsilon$. For $\mu$ sufficiently close to $\mu_*$, we can identify the H\'enon map $H_{\mu, \nu}$ as a pair in $\mathcal{D}_2(\Omega, \Gamma, \epsilon)$ as follows:
\begin{equation}\label{eq:henonpair}
\Sigma_{H_{\mu,\nu}} := \Lambda(H_{\mu,\nu}^2|_{\Omega_{\mu, \nu}}, H_{\mu,\nu}|_{\Gamma_{\mu, \nu}}).
\end{equation}
where
\begin{displaymath}
\Omega_{\mu, \nu} := s_{H_{\mu, \nu}}(\Omega) = \pi_1 H_{\mu, \nu}(0)\cdot \Omega
\hspace{5mm} \text{and} \hspace{5mm}
\Gamma_{\mu, \nu} := s_{H_{\mu, \nu}}(\Gamma) = \pi_1 H_{\mu, \nu}(0)\cdot \Gamma.
\end{displaymath}

The following is a consequence of Theorem \ref{quadratic} and \ref{renormalization hyperbolicity}:

\begin{cor} \label{henon intersect}
The two parameter family $\{\Sigma_{H_{\mu,\nu}}\}_{\mu, \nu}$ intersects the stable manifold $W^s(\iota(\zeta_*)) \subset \mathcal{D}_2(\Omega, \Gamma, \epsilon)$ of the fixed point $\iota(\zeta_*)$ for $\mathbf{R}$.
\end{cor}

\begin{figure}[h]
\centering
\includegraphics[scale=0.4]{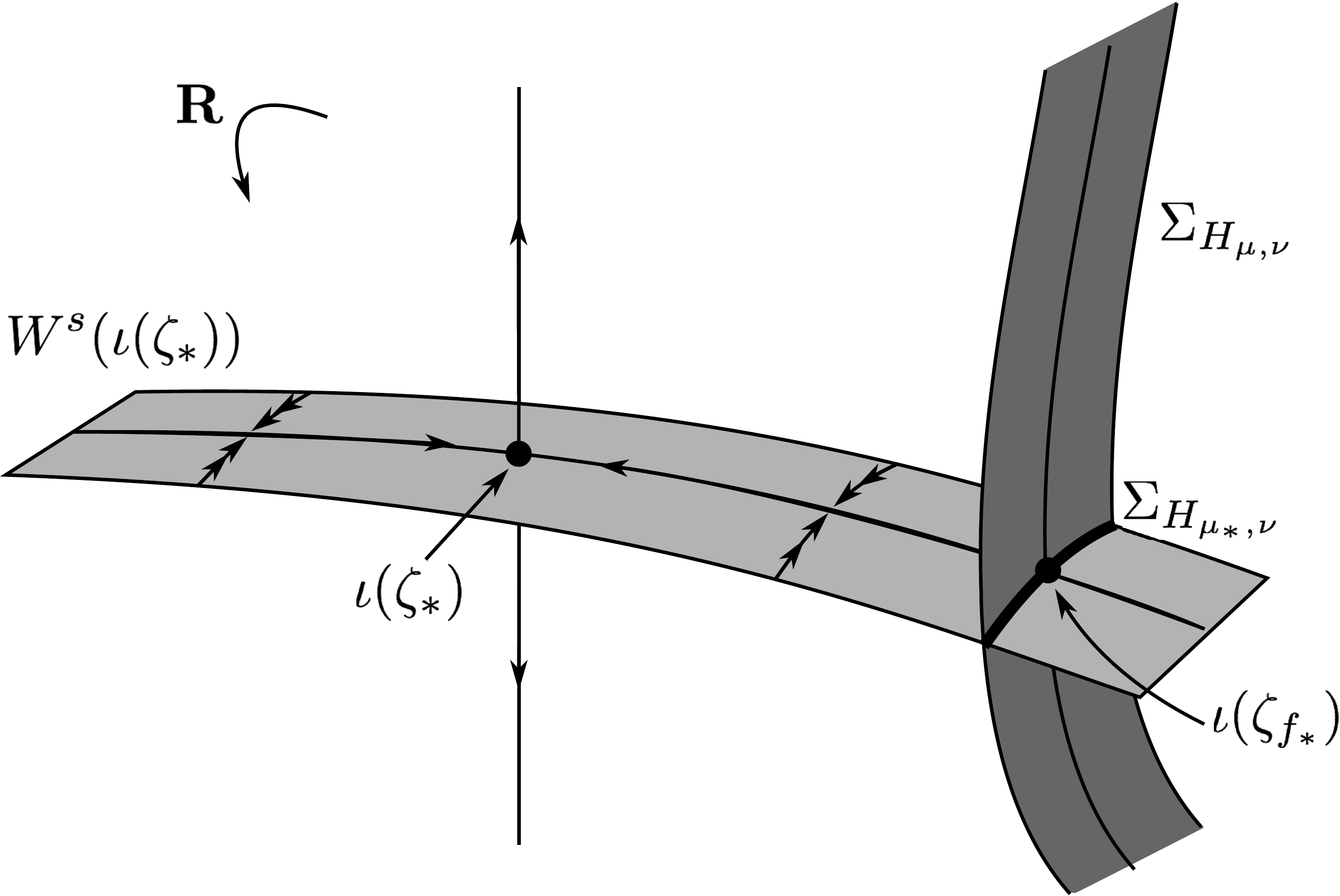}
\caption{The stable manifold $W^s(\iota(\zeta_*))$ of the fixed point $\iota(\zeta_*)$ for the 2D-renormalization operator $\mathbf{R}$. Every pair in $W^s(\iota(\zeta_*))$ converge super-exponentially fast to the space of degenerate one-dimensional pairs. Every degenerate one-dimensional pair in $W^s(\iota(\zeta_*))$ converge exponentially fast to $\iota(\zeta_*)$, at a rate given by Theorem \ref{1d renormalization}. By Theorem \ref{siegel henon}, the family of H\'enon maps $\{\Sigma_{H_{\mu, \nu}}\}_{\mu, \nu}$ intersect $W^s(\zeta_*)$ at the golden-mean semi-Siegel H\'enon maps $\{\Sigma_{H_{\mu_*, \nu}}\}_\nu$. Note that $\Sigma_{H_{\mu_*, 0}} = \iota(\zeta_{f_*})$, where $\zeta_{f_*}$ is the pair representation of the golden-mean quadratic Siegel polynomial $f_*$. Compare with Figure \ref{fig:1dhyperbolicity}.}
\label{fig:2dhyperbolicity}
\end{figure}

\section{Convergence of the Microscope Maps}\label{sec:composition}

Let $\Sigma = (A, B)$ be a pair contained in the stable manifold $W^s(\iota(\zeta_*))$ of the 2D-renormalization fixed point $\iota(\zeta_*)$. Moreover, assume that $\Sigma$ is a commuting pair. Set
\begin{displaymath}
\Sigma_n = (A_n, B_n) := \mathbf{R}^n(\Sigma),
\end{displaymath}
where
\begin{displaymath}
A_n(x,y) = \begin{bmatrix}
a_n(x,y) \\
h_n(x,y)
\end{bmatrix}
\hspace{5mm} \text{and} \hspace{5mm}
B_n(x,y) = \begin{bmatrix}
b_n(x,y) \\
x
\end{bmatrix}.
\end{displaymath}
Let
\begin{displaymath}
\eta_n(x) := a_n(x, 0)
\hspace{5mm} , \hspace{5mm}
\xi_n(x) := b_n(x,0)
\hspace{5mm} \text{and} \hspace{5mm}
\zeta_n := (\eta_n, \xi_n).
\end{displaymath}
Then by Theorem \ref{renormalization hyperbolicity}, we have
\begin{equation}\label{eq:embedding estimate}
\|\Sigma_{n+1} - \iota(\zeta_n)\| < O(\epsilon^{2^n}).
\end{equation}

Denote
\begin{displaymath}
(a_n)_y(x) := a_n(x,y),
\end{displaymath}
and let
\begin{displaymath}
H_n(x,y) :=
\begin{bmatrix}
(a_n)_y^{-1}(x) \\
y
\end{bmatrix}
\end{displaymath}
be the non-linear changes of coordinates given in \eqref{eq:nonlinear change of coord}. If
\begin{displaymath}
\tilde{B}_{n+1} := H_n^{-1} \circ B_n \circ A_n \circ H_n
\hspace{5mm} \text{and} \hspace{5mm}
\tilde{\beta}_{n+1} := \pi_1 \tilde{B}_{n+1},
\end{displaymath}
then by \eqref{eq:renormalization estimates}, the map $\tilde{\beta}_{n+1}$ has a unique critical point $c_n$ near $0$. Define
\begin{displaymath}
T_n(x,y):=(x+c_n, y),
\end{displaymath}
and let
\begin{displaymath}
s_n(x,y) := (\lambda_n x, \lambda_n y)
\hspace{5mm} , \hspace{5mm}
|\lambda_n| < 1
\end{displaymath}
be the scaling map so that if
\begin{equation} \label{eq:change of coordinates}
\Phi_n(x,y) = \begin{bmatrix}
\phi_n(x) \\
\lambda_n y
\end{bmatrix} :=  H_n \circ T_n \circ s_n(x,y),
\end{equation}
then we have
\begin{displaymath}
A_{n+1} = \Phi_n^{-1} \circ B_n \circ A_n^2 \circ \Phi_n
\hspace{5mm} \text{and} \hspace{5mm}
B_{n+1} = \Phi_n^{-1} \circ B_n \circ A_n \circ \Phi_n.
\end{displaymath}

The following is a direct consequence of Theorem \ref{1d renormalization} and \ref{renormalization hyperbolicity}:

\begin{cor} \label{convergence}
As $n \to \infty$, we have the following convergences (each of which occurs at a geometric rate):
\begin{enumerate}[label=(\roman{*})]
\item $\zeta_n =(\eta_n, \xi_n) \to \zeta_* = (\eta_*, \xi_*)$;
\item $\lambda_n \to \lambda_*$, where $\lambda_*$ is the universal scaling constant given in Theorem \ref{1d renormalization}; and
\item $\phi_n \to g_*$ and $\Phi_n \to G_*$, where 
\begin{displaymath}
G_*(x,y) =
\begin{bmatrix}
g_*(x) \\
\lambda_* y
\end{bmatrix}
:=
\begin{bmatrix}
\eta_*^{-1}(\lambda_* x) \\
\lambda_* y
\end{bmatrix}.
\end{displaymath}
\end{enumerate}
\end{cor}

\begin{prop}\label{kappa}
The map $g_* : Z \to Z$ given in Corollary \ref{convergence} has an attracting fixed point at $1$ with multiplier $\lambda_*^2$.
\end{prop}

\begin{proof}
Recall
\begin{displaymath}
\lambda_*:= \eta_* \circ \xi_*(0) = \eta_*(1).
\end{displaymath}
Immediately, we see that the map
\begin{displaymath}
g_*(x) := \eta_*^{-1}(\lambda_*x)
\end{displaymath}
fixes the point $1$. Moreover, since $g_*(Z) \Subset Z$, this fixed point must be attracting.

Since $\xi_*$ has a critical point at $0$, we may write
\begin{displaymath}
\xi_*(x) = 1+ c_* x^2 + O(|x|^3).
\end{displaymath}
for some $c_* \in \mathbb{C}$. Thus,
\begin{displaymath}
\lambda_* \xi_*(x) = \lambda_*+ \lambda_*c_*x^2 +O(|x|^3)
\hspace{5mm} \text{and} \hspace{5mm}
\xi_*(\lambda_* x) = 1+ \lambda_*^2 c_*x^2 +O(|x|^3).
\end{displaymath}
Since $\zeta_* = (\eta_*, \xi_*)$ is a renormalization fixed point, we have
\begin{displaymath}
\lambda_* \xi(x) = \eta_* \circ \xi_*(\lambda_* x) = \lambda_* + \eta_*'(1) \lambda_*^2 c_*x^2 + O(|x|^3).
\end{displaymath}
Therefore
\begin{displaymath}
\eta_*'(1) = \lambda_*^{-1},
\end{displaymath}
and we conclude
\begin{displaymath}
g_*'(1) = \frac{\lambda_*}{\eta_*'(1)} = \lambda_*^2.
\end{displaymath}
\end{proof}

Let $t_1(x) := x+1$ denote the unit translation by $1$. Define
\begin{displaymath}
\mathbf{g}_* := t_1^{-1} \circ g_* \circ t_1.
\end{displaymath}
Then since $\mathbf{g}_*$ has an attracting fixed point at $0$, the sequence $\lambda_*^{-2n}\mathbf{g}_*^n$ converges to the linearizing map $u_* : t_1^{-1}(Z) \to \mathbb{C}$ for $\mathbf{g}_*$ at $0$ as $n \to \infty$.

Consider the map $\phi_n$ given in \eqref{eq:change of coordinates}. Define
\begin{displaymath}
\phi^0_k := \text{Id}
\hspace{5mm} \text{and} \hspace{5mm}
\phi^n_k := \phi_k \circ \phi_{k+1} \circ \ldots{} \circ \phi_{k+n-1}
\hspace{5mm} \text{for } n \geq 1.
\end{displaymath}
Let
\begin{displaymath}
\psi_n := t_1^{-1} \circ \phi_n \circ t_1.
\end{displaymath}
By Corollary \ref{convergence}, the sequence $\psi_n$ converges to $\mathbf{g}_*$ geometrically fast as $n$ goes to $\infty$. Define
\begin{displaymath}
\psi^0_k := \text{Id},
\hspace{5mm} \text{and} \hspace{5mm}
\psi^n_k :=\psi_k \circ \psi_{k+1} \circ \ldots{} \circ \psi_{k+n-1}=t_1^{-1} \circ \phi^n_k \circ t_1
\hspace{5mm} \text{for } n \geq 1.
\end{displaymath}

\begin{lem}\label{nonlinearity}
We have
\begin{displaymath}
\|\lambda_*^{-2n}\psi^n_k- \lambda_*^{-2n}\mathbf{g}_*^n\| < C n \rho^{k+n-1},
\end{displaymath}
for some uniform constants $C$ and $\rho<1$. Hence, we have the following convergence:
\begin{displaymath}
\lambda_*^{-2n} \psi^n_k \to u_*
\hspace{5mm} \text{as} \hspace{5mm}
n \to \infty.
\end{displaymath}
\end{lem}

\begin{proof}
Denote $m := k+n-1$, and let
\begin{displaymath}
d_m := \psi_m'(0).
\end{displaymath}
For $k \leq i < m$, define
\begin{displaymath}
d_i := \lambda_*^{-2}\psi_i'(0)d_{i+1}.
\end{displaymath}
Let
\begin{displaymath}
e_i(x) := \lambda_*^{-2}d_i x.
\end{displaymath}
Since $\psi_n$ converges to $\mathbf{g}_*$ geometrically fast as $n$ goes to $\infty$, we have
\begin{displaymath}
\lambda_*^{-2}d_i = O(1),
\end{displaymath}
and we may express
\begin{displaymath}
d_i^{-1}\psi_i \circ e_{i+1} = \lambda_*^{-2}\mathbf{g}_*+ E_i
\end{displaymath}
where
\begin{displaymath}
\|E_i\|<C_1 \rho_1^i,
\end{displaymath}
for some uniform constants $C_1$ and $\rho_1<1$. Note that we have
\begin{displaymath}
E_i'(0) = 0.
\end{displaymath}
By Cauchy-estimates, there exists a uniform constant $C_2$ such that
\begin{displaymath}
\|E_i(x)\| < C_2 \rho_1^i |x|^2.
\end{displaymath}
Let
\begin{displaymath}
\rho = \max\{\rho_1,|\lambda_*^2|\}.
\end{displaymath}

Observe
\begin{align*}
\lambda_*^{-2n} \psi^i_k \circ (\psi_i \circ e_{i+1}) \circ \mathbf{g}_*^{m-i} &= \lambda_*^{-2n} \psi^i_k \circ (e_i \circ \mathbf{g}_* + d_i E_i) \circ \mathbf{g}_*^{m-i}\\
&= \lambda_*^{-2n} \psi^i_k \circ e_i \circ \mathbf{g}_*^{m-i+1} + O(\|\lambda_*^{-2n} \psi^i_k \circ d_i E_i \circ \mathbf{g}_*^{m-i}\|)\\
&= \lambda_*^{-2n} \psi^i_k \circ e_i \circ \mathbf{g}_*^{m-i+1} + O(\rho^{-n} \rho^{i-k}\rho \rho^i \rho^{2(m-i)})\\
&= \lambda_*^{-2n} \psi^i_k \circ e_i \circ \mathbf{g}_*^{m-i+1} + O(\rho^m)
\end{align*}
The desired inequality follows.
\end{proof}

Define
\begin{equation}\label{eq:composition}
\Phi^0_k := \text{Id},
\hspace{5mm} \text{and} \hspace{5mm}
\Phi^n_k := \Phi_k \circ \Phi_{k+1} \circ \ldots{} \circ \Phi_{k+n-1}
\hspace{5mm} \text{for} \hspace{5mm}
n \geq 1.
\end{equation}
Also, define
\begin{displaymath}
\Lambda^0_k := \text{Id},
\hspace{5mm} \text{and} \hspace{5mm}
\Lambda^n_k := \begin{bmatrix}
\lambda_*^{2n} x \\
\lambda_k \cdot \lambda_{k+1} \cdot \ldots \cdot \lambda_{k+n -1} y
\end{bmatrix}
\hspace{5mm} \text{for} \hspace{5mm}
n \geq 1.
\end{displaymath}
Lastly, let
\begin{displaymath}
T_1(x,y):= \begin{bmatrix}
x+1 \\
y
\end{bmatrix}
\hspace{5mm} \text{and} \hspace{5mm}
U_*(x,y):=\begin{bmatrix}
u_*(x)\\
y
\end{bmatrix}.
\end{displaymath}

\begin{cor}\label{composition convergence}
We have the following convergence:
\begin{displaymath}
(\Lambda_k^n)^{-1} \circ T_1^{-1}\circ \Phi_k^n \circ T_1 \to U_*
\hspace{5mm} \text{as} \hspace{5mm}
n \to \infty.
\end{displaymath}
\end{cor}

\section{The Combinatorics of Golden-Mean Rotation}\label{sec:combin of rot}

In this section, we study the combinatorics of the two-dimensional renormalization defined in Section \ref{sec:renormalization}. To simplify our analysis, we model the dynamics of almost commuting pairs by rigid interval exchange maps of the inverse golden-mean rotation type.

\subsection*{Pre-renormalization operator for golden-mean rotation} Consider $s \in (0, \theta_*]$ and $t \in (0, 1]$ such that $s/t = \theta_* = (\sqrt{5}-1)/2$. Let
\begin{displaymath}
I = [1 - t - s, 1-t]
\hspace{5mm} \text{and} \hspace{5mm}
J=[1-t,1].
\end{displaymath}
Note that we have
\begin{displaymath}
|I| = s < t = |J|.
\end{displaymath}
Define the maps $S : J \to I \cup J$ and $T : I \to J$ as follows:
\begin{displaymath}
S(x) := x - s
\hspace{5mm} \text{and} \hspace{5mm}
T(x) := x + t.
\end{displaymath}
The action of the pair of maps $R = (S|_J, T|_I)$ on the interval $I \cup J$ represents the rigid rotation of the circle by the angle $\theta_*$.

\begin{figure}[h]
\centering
\includegraphics[scale=0.4]{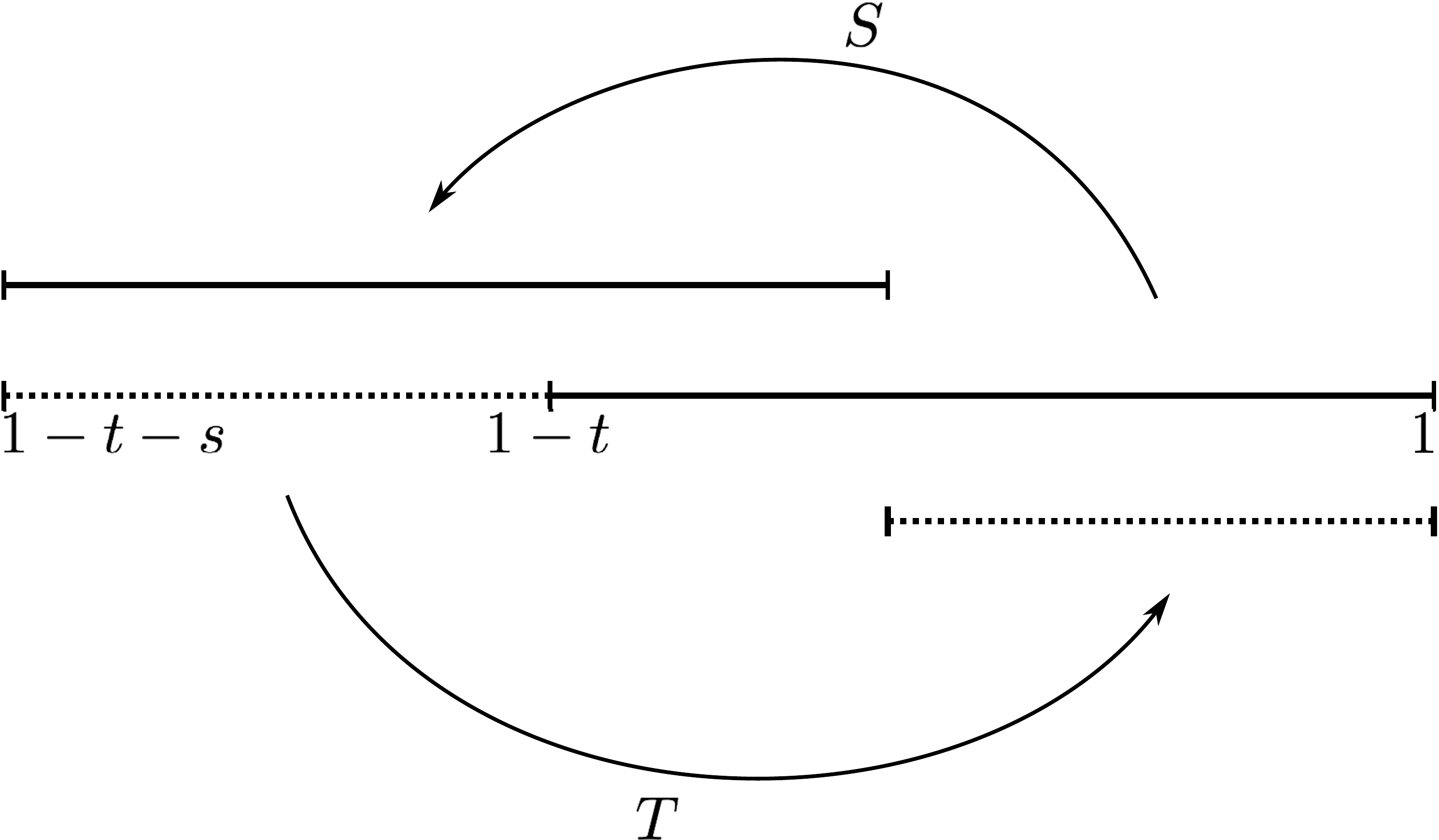}
\caption{A rigid rotation pair $R = (S|_J, T|_I)$, where $I=[1-t-s, 1-t]$ and $J=[1-t,1]$.}
\label{fig:rotationpair}
\end{figure}

We define the {\it pre-renormalization} $p\mathcal{R}(R)$ of $R$ as follows. Let
\begin{displaymath}
s' := 2s-t \in (0, s)
\hspace{5mm} \text{and} \hspace{5mm}
t' := t-s \in (0, t).
\end{displaymath}
Then define
\begin{displaymath}
p\mathcal{R}(R) = (T \circ S^2|_{J'}, T\circ S|_{I'}),
\end{displaymath}
where $I' = [1-t'-s', 1-t']$ and $J' = [1-t', 1]$. Similar to before, we have $s'/t'=\theta_*$, and
\begin{displaymath}
|I'| = s' < t' = |J'|.
\end{displaymath}
Hence, the action of $p\mathcal{R}(R)$ on the interval $I' \cup J'$ represents the rigid rotation of the circle by the angle $\theta_*$

\begin{figure}[h]
\centering
\includegraphics[scale=0.4]{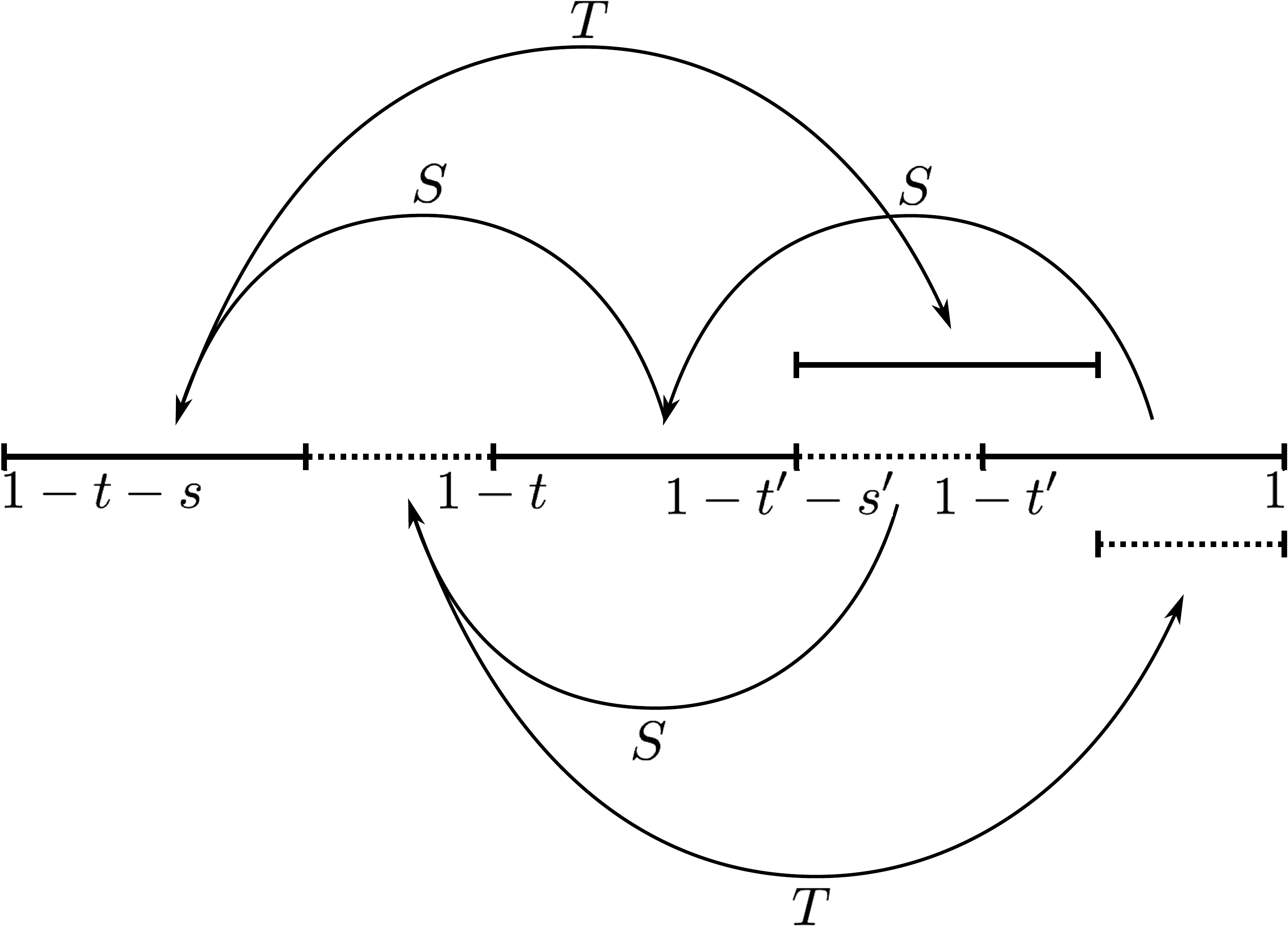}
\caption{The pre-renormalization $p\mathcal{R}(R)= (T \circ S^2|_{J'}, T\circ S|_{I'})$, where $I' = [1-t'-s', 1-t']$ and $J' = [1-t', 1]$.}
\label{fig:prerenormalization}
\end{figure}

\subsection*{Dynamical partitions} Set
\begin{displaymath}
s_0 := \theta_*
\hspace{5mm} , \hspace{5mm}
t_0 := 1
\hspace{5mm} , \hspace{5mm}
I_0 := [-\theta_*, 0]
\hspace{5mm} \text{and} \hspace{5mm}
J_0 := [0, 1].
\end{displaymath}
Define
\begin{equation}\label{eq:rigid pair}
S_0(x):= x -\theta_*
\hspace{5mm} ,\hspace{5mm}
T_0(x):= x + 1,
\end{equation}
and consider the pair $R_0 = (S_0|_{J_0}, T_0|_{I_0})$ acting on the interval $[-\theta_*, 1]$.

For $n \in \mathbb{N}$, denote the $n$th pre-renormalization of $R_0$ by
\begin{displaymath}
R_n = (S_n|_{J_n}, T_n|_{I_n}) := p\mathcal{R}^n(R_0),
\end{displaymath}
where
\begin{equation}\label{base dynamic interval}
I_n = [1 - t_n - s_n, 1-t_n]
\hspace{5mm} \text{and} \hspace{5mm}
J_n=[1-t_n,1],
\end{equation}
and 
\begin{displaymath}
S_n(x):= x - s_n
\hspace{5mm} \text{and} \hspace{5mm}
T_n(x):= x + t_n.
\end{displaymath}
Then we have
\begin{equation}\label{rigid ren angle}
\frac{s_n}{t_n} = \theta_*.
\end{equation}

\begin{notn}\label{rotationcombinatorics}
For $n \in \mathbb{N}$, consider an $n$-tuple
\begin{displaymath}
\overline{\omega} = (\alpha_{n-1}, \ldots{}, \alpha_0)
\end{displaymath}
constructed inductively from $i = n-1$ to $i = 0$ as follows:
\begin{enumerate}[label=(\roman{*})]
\item Choose $\alpha_{n-1} \in \{0, 1, 2\}$.
\item If $\alpha_{i+1}=2$, then choose $\alpha_i \in \{0,1\}$.
\item If $\alpha_{i+1}$ was chosen from $\{0,1\}$, and $\alpha_{i+1}=1$, then choose $\alpha_i \in \{0,1\}$.
\item Otherwise, choose $\alpha_i \in \{0,1,2\}$.
\end{enumerate}
Denote the set of all $n$-tuples constructed as above by $\mathcal{J}_n$. For $n=0$, we define $\mathcal{J}_0 := \{(0)\}$.

We also denote by $\mathcal{I}_n$ the set of all $n$-tuples
\begin{displaymath}
\overline{\gamma} = (\beta_{n-1}, \ldots{}, \beta_0)
\end{displaymath}
constructed identically as for $\mathcal{J}_n$, except step (i) is replaced by
\begin{enumerate}[label=(\roman{*}')]
\item Choose $\beta_{n-1} \in \{0, 1\}$.
\end{enumerate}
\end{notn}

\begin{lem}\label{rigid bead spread}
Let
\begin{displaymath}
\overline{\omega} = (\alpha_{n-1}, \ldots{}, \alpha_0) \in \mathcal{J}_n
\hspace{5mm} \text{and} \hspace{5mm}
\overline{\gamma} = (\beta_{n-1}, \ldots{}, \beta_0) \in \mathcal{I}_n.
\end{displaymath}
Denote
\begin{displaymath}
R_0^{\overline{\omega}}:=  S_0^{\alpha_0}|_{J_0} \ldots{} \circ S_{n-1}^{\alpha_{n-1}}|_{J_{n-1}}
\hspace{5mm} \text{and} \hspace{5mm}
R_0^{\overline{\gamma}}:=  S_0^{\beta_0}|_{J_0} \ldots{} \circ S_{n-1}^{\beta_{n-1}}|_{J_{n-1}}.
\end{displaymath}
Then $R_0^{\overline{\omega}}$ and $R_0^{\overline{\gamma}}$ are well-defined on $J_n$ and $I_n$ respectively.
\end{lem}

\begin{lem}\label{rigid renormalization spread}
Let
\begin{displaymath}
\overline{\omega}^{\emph{max}}_n := (2, 1, 1, \ldots, 1) \in \mathcal{J}_n
\hspace{5mm} \text{and} \hspace{5mm}
\overline{\gamma}^{\emph{max}}_n := (1, 1, \ldots, 1) \in \mathcal{I}_n.
\end{displaymath}
Then we have
\begin{displaymath}
R_n = p\mathcal{R}^n(R_0) = (T_0 \circ R_0^{\overline{\omega}^{\emph{max}}_n}|_{J_n}, T_0 \circ R_0^{\overline{\gamma}^{\emph{max}}_n}|_{I_n}).
\end{displaymath}
\end{lem}

\begin{lem}\label{dynamic partition}
Define
\begin{displaymath}
\mathcal{P}_n := \{R_0^{\overline{\omega}}(J_n) \, | \, \overline{\omega} \in \mathcal{J}_n\}
\end{displaymath}
and
\begin{displaymath}
\mathcal{Q}_n := \{R_0^{\overline{\gamma}}(I_n) \, | \, \overline{\gamma} \in \mathcal{I}_n\}.
\end{displaymath}
Then $\mathcal{P}_n \cup \mathcal{Q}_n$ forms a cover of $[-\theta_*, 1]$ such that its members are disjoint except at the endpoints. The collection $\mathcal{P}_n \cup \mathcal{Q}_n$ is called the \emph{$n$th dynamical partition} of $[-\theta_*, 1]$.
\end{lem}

\begin{figure}[h]
\centering
\includegraphics[scale=0.4]{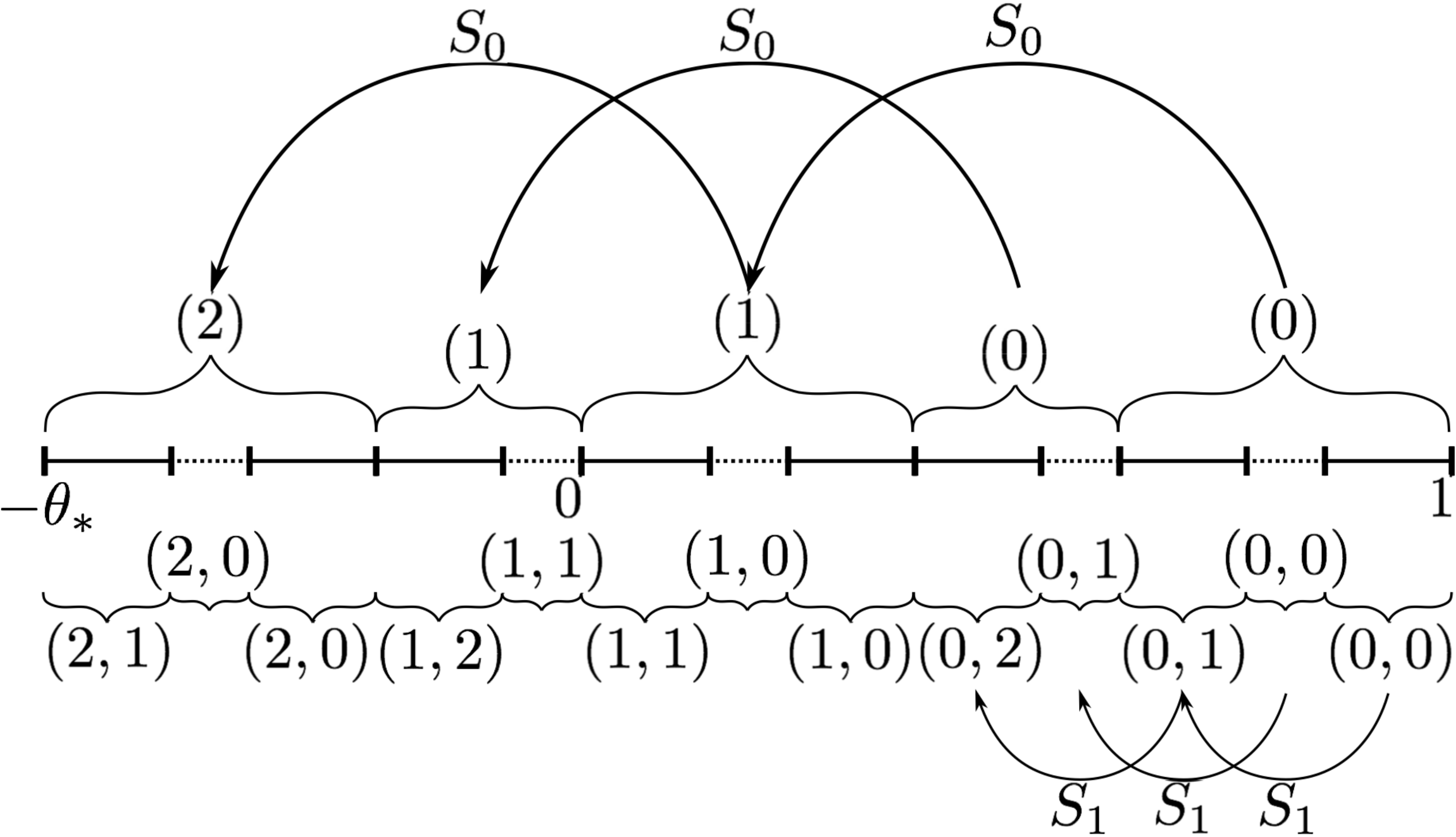}
\caption{The elements of the $1$st and $2$nd dynamic partitions $\mathcal{P}_1 \cup \mathcal{Q}_1$ and $\mathcal{P}_2 \cup \mathcal{Q}_2$.}
\label{fig:dynamicpartition}
\end{figure}

\begin{lem}\label{element deconstruction}
For $n \geq 0$, let $U \in \mathcal{P}_n$. Listing in order from left to right, the element $U$ consists of one element in $\mathcal{P}_{n+1}$, one element in $\mathcal{Q}_{n+1}$, and another element in $\mathcal{P}_{n+1}$. 

Similarly, let $V \in \mathcal{Q}_n$. Listing in order from left to right, the element $V$ consists of one element in $\mathcal{P}_{n+1}$, and one element in $\mathcal{Q}_{n+1}$. 
\end{lem}

\begin{lem}\label{closestreturns}
Let $\{q_n\}_{n=0}^\infty \subset \mathbb{N}$ be the \emph{Fibonacci sequence} defined by the following inductive relation:
\begin{displaymath}
q_0=1
\hspace{5mm} , \hspace{5mm}
q_1=1
\hspace{5mm} \text{and} \hspace{5mm}
q_{n+1} = q_n + q_{n-1}
\hspace{5mm} \text{for}  \hspace{5mm}
n \geq 1.
\end{displaymath}
Then $q_{2n+1} = |\mathcal{J}_n|$ and $q_{2n} = |\mathcal{I}_n|$.
\end{lem}

Define
\begin{displaymath}
Q_n := \bigcup_{V \in \mathcal{Q}_n} V.
\end{displaymath}
By Lemma \ref{dynamic partition} and \ref{closestreturns}, the set $Q_n$ is a union of $q_{2n}$ intervals of length $s_n$. The following result shows that these intervals are well-distributed over $[-\theta_*, 1]$, in the sense that the average of any sufficiently well-behaved function on $[-\theta_*, 1]$ is approximately equal to its average on $Q_n$. Moreover, the error is of the same order of magnitude as $s_n$.

\begin{prop}\label{int prop}
Let $f : [-\theta_*, 1] \to \mathbb{C}$ be a piecewise-smooth function with finitely many discontinuities, whose derivative is bounded by $M$. Then we have
\begin{displaymath}
 \frac{1}{q_{2n} s_n}\int_{Q_n} f(x) \, dx = \frac{1}{1+\theta_*}\int_{-\theta_*}^1 f(x)\, dx + O(Ms_n).
\end{displaymath}
\end{prop}

\begin{proof}
Denote
\begin{displaymath}
m_n:=  \frac{1+\theta_*}{q_{2n}s_n} > 1,
\end{displaymath}
and let $u_n : Q_n \to [-\theta_*, 1]$ be the unique surjective map satisfying the following two properties:
\begin{enumerate}[label=(\roman*)]
\item The restriction of $u_n$ to any element $V \in \mathcal{Q}_n$ is an affine map of the form
\begin{displaymath}
u_n|_V(x) = m_n x + b_V
\end{displaymath}
for some $b_V \in \mathbb{R}$.
\item For any $x, y \in Q_n$, if $x < y$, then $u_n(x) \leq u_n(y)$.
\end{enumerate}
Then we have
\begin{displaymath}
\int_{-\theta_*}^1 f(u_n^{-1}(x)) \, dx = m_n \int_{Q_n} f(x) \, dx.
\end{displaymath}

\begin{figure}[h]
\centering
\includegraphics[scale=0.4]{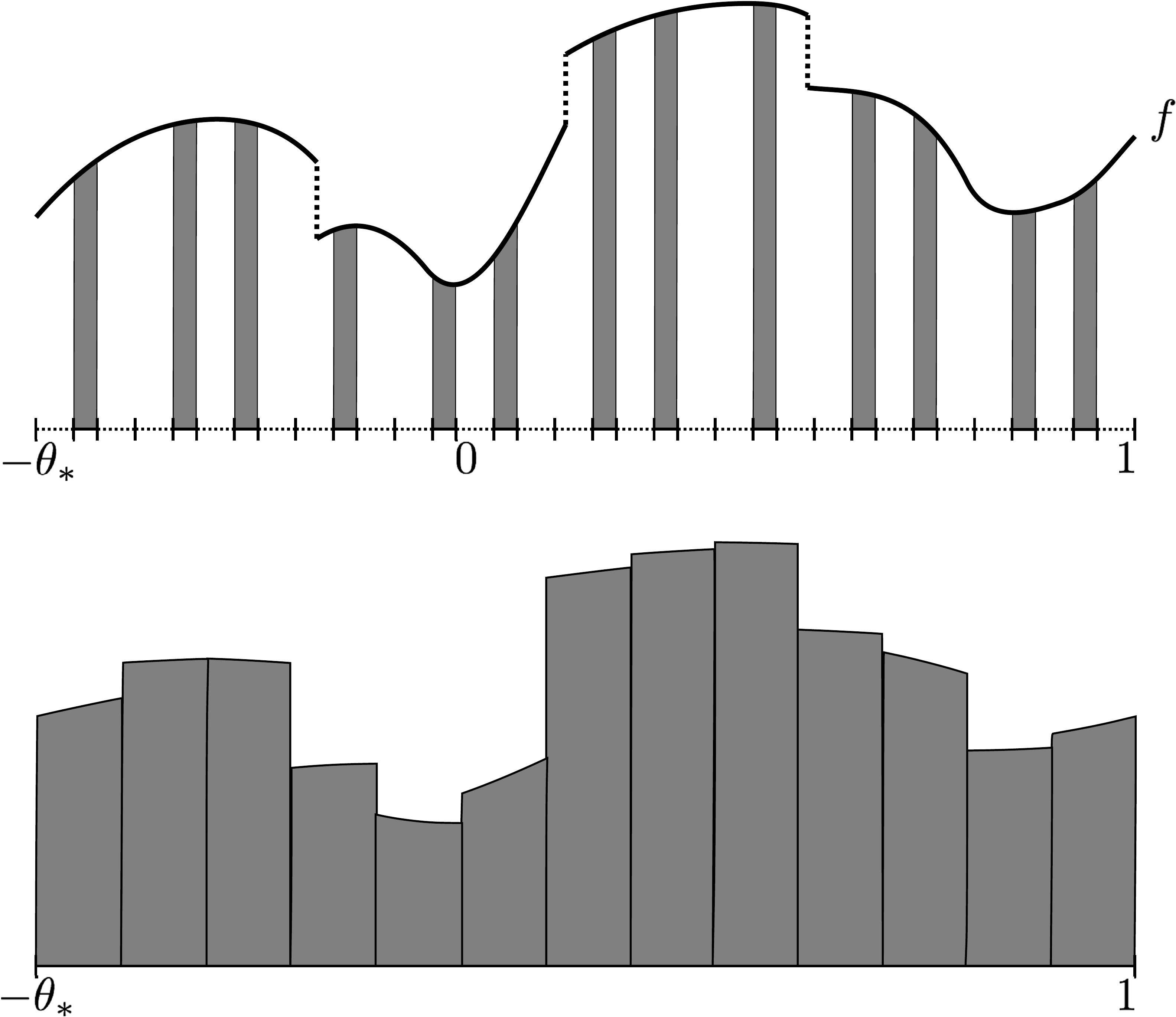}
\caption{The integrals $\int_{-\theta_*}^1 f(x) \, dx$ (top), $\int_{Q_n} f(x) \, dx$ (top, in grey), and $\int_{-\theta_*}^1 f(u_n^{-1}(x)) \, dx = m_n \int_{Q_n} f(x) \, dx$ (bottom).}
\label{fig:integral}
\end{figure}

Write
\begin{displaymath}
m_n\int_{Q_n} f(x) \, dx = \int_{-\theta_*}^1 f(x)\, dx + E_n,
\end{displaymath}
where
\begin{displaymath}
E_n := \int_{-\theta_*}^1 f(u_n^{-1}(x)) \, dx - \int_{-\theta_*}^1 f(x)\, dx.
\end{displaymath}
Observe that
\begin{align}
\left| E_n \right | &\leq \int_{-\theta_*}^1 \left|f(u_n^{-1}(x))- f(x) \right | dx \nonumber\\
&\leq M \int_{-\theta_*}^1 \left| u_n^{-1}(x) - x \right | dx \label{int proof 0}
\end{align}
To estimate \eqref{int proof 0}, we need to find a bound on the displacement of points under $u_n$.

Consider the $k$th dynamic partition $\mathcal{P}_k \cup \mathcal{Q}_k$ for $0 \leq k \leq n-1$. The map $u_n$ acts by eliminating the elements that belong to $\mathcal{P}_n$ and stretching the elements that belong to $\mathcal{Q}_n$ by a factor of $m_n$. Denote the change in size under $u_n$ of each element in $\mathcal{P}_k$ and $\mathcal{Q}_k$ by $\tau^k_n$ and $\sigma^k_n$ respectively:
\begin{displaymath}
\tau^k_n := |u_n(U_k)| - |U_k| \hspace{5mm} \text{for any} \hspace{5mm} U_k \in \mathcal{P}_k,
\end{displaymath}
and
\begin{displaymath}
\sigma^k_n := |u_n(V_k)| - |V_k| \hspace{5mm} \text{for any} \hspace{5mm} V_k \in \mathcal{Q}_k.
\end{displaymath}

\begin{figure}[h]
\centering
\includegraphics[scale=0.4]{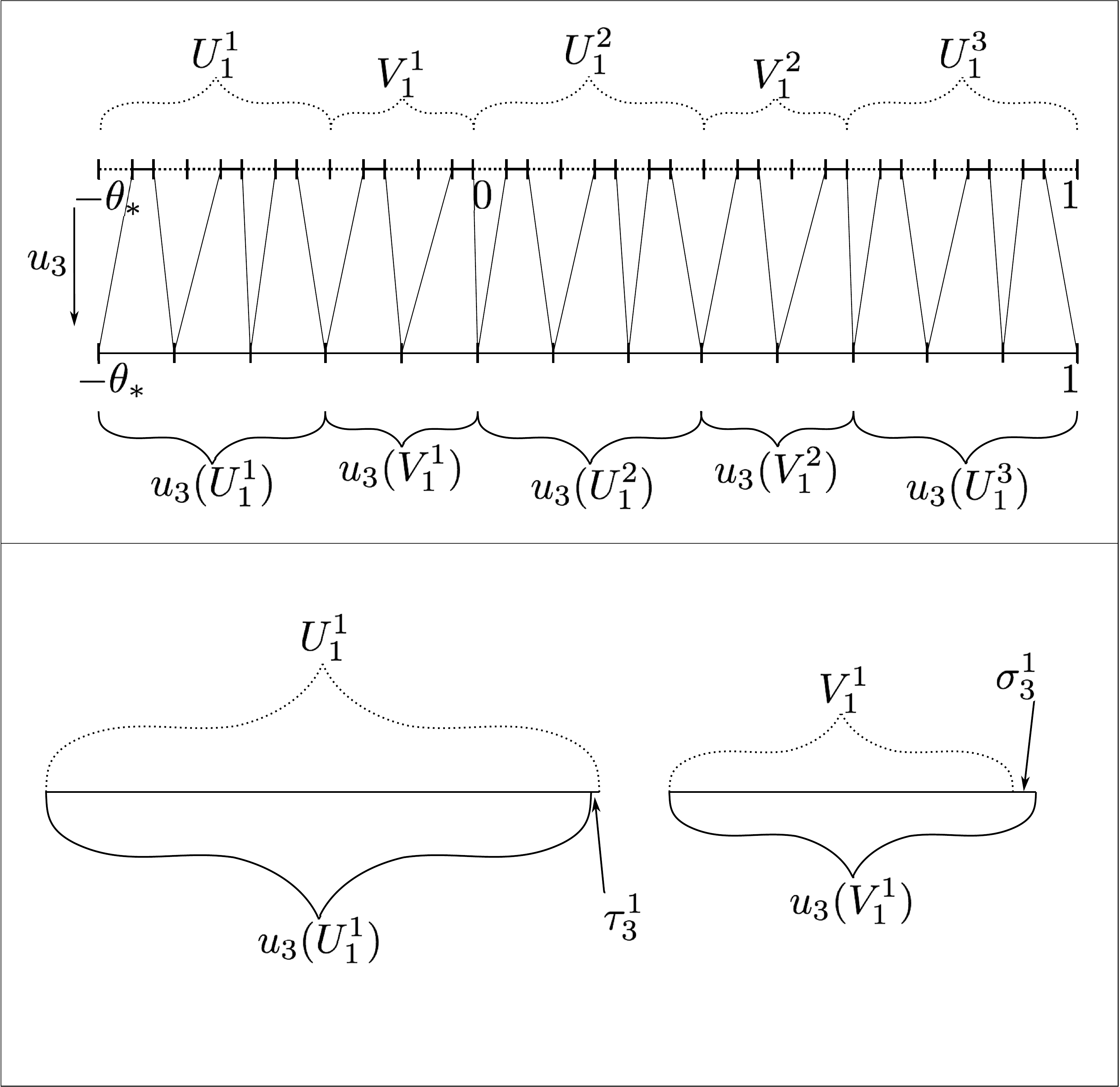}
\caption{The images of the elements $U_1^1, U_1^2, U_1^3 \in \mathcal{P}_1$ and $V_1^1, V_1^2 \in \mathcal{Q}_1$ under the piece-wise affine map $u_3$. Under $u_3$, the elements of $\mathcal{P}_3$ are discarded, and the elements of $\mathcal{Q}_3$ are stretched by a factor of $m_3$. As a result, the elements in $\mathcal{P}_1$ shrink by $\tau^1_3$, while the elements in $\mathcal{Q}_1$ expand by $\sigma^1_3$.}
\label{fig:shrinkexpand}
\end{figure}

By Lemma \ref{element deconstruction}, we have
\begin{equation}\label{int proof 1}
\tau^{n-1}_n:= (m_n-1)s_n - 2t_n < 0
\hspace{5mm} \text{and} \hspace{5mm}
\sigma^{n-1}_n  := (m_n-1)s_n - t_n> 0.
\end{equation}
It follows that
\begin{displaymath}
|\tau^{n-1}_n|<t_n
\hspace{5mm} \text{and} \hspace{5mm}
|\sigma^{n-1}_n| < t_n.
\end{displaymath}
Likewise, for $0 \leq k < n-1$, we have
\begin{equation}\label{int proof 2}
\tau^k_n := \sigma^{k+1}_n+ 2\tau^{k+1}_n
\hspace{5mm} \text{and} \hspace{5mm}
\sigma^k_n := \sigma^{k+1}_n+\tau^{k+1}_n.
\end{equation}
Note that the pairs $\{\tau^{k+1}_n, \sigma^{k+1}_n\}$ and $\{\tau^k_n, \sigma^k_n\}$ each have opposite signs. Hence, \eqref{int proof 2} implies that the pairs $\{\tau^k_n, \tau^{k+1}_n\}$ and $\{\sigma^k_n, \sigma^{k+1}_n\}$ each have the same sign, and
\begin{displaymath}
|\sigma^k_n| < |\sigma^{k+1}_n|
\hspace{5mm} \text{and} \hspace{5mm}
|\tau^k_n| < |\tau^{k+1}_n|.
\end{displaymath}
Thus, by \eqref{int proof 1}, we have
\begin{displaymath}
\sigma^k_n > 0
\hspace{5mm} \text{and} \hspace{5mm}
\tau^k_n <0
\hspace{5mm} \text{for all} \hspace{5mm}
0 \leq k \leq n-1.
\end{displaymath}

Let $\hat{x} \in Q_n$ be a point of maximum displacement under $u_n$:
\begin{displaymath}
\text{i.e.} \hspace{5mm} \max_{-\theta_* \leq x \leq 1} |u_n^{-1}(x)-x| = |\hat{x} - u_n(\hat{x})|.
\end{displaymath}
To obtain the desired estimate on \eqref{int proof 0}, we will find a bound on the displacement of $\hat{x}$ under $u_n$.

The interval $[-\theta_*, 0]$ is occupied by the element $I_0$ in $\mathcal{Q}_0$, and the interval $[0, 1]$ is occupied by the element $J_0$ in $\mathcal{P}_0$. By Lemma \ref{element deconstruction}, listing from left to right, the first dynamic partition $\mathcal{P}_1 \cup \mathcal{Q}_1$ consists of $U^1_1 \in \mathcal{P}_1$, $V^1_1 \in \mathcal{Q}_1$, $U^2_1 \in \mathcal{P}_1$, $V^2_1 \in \mathcal{Q}_1$, and  $U^3_1 \in \mathcal{P}_1$. Note that for $x_1 \in U^1_1 \cap Q_n$, we have $x_1 + s_0 \in U^2_1 \cap Q_n$ and $x_1+2s_0 \in U^3_1 \cap Q_n$, and
\begin{equation}\label{intproof1eqn1}
u_n(x_1 + ls_0) = u_n(x_1)+l\sigma^0_n
\hspace{5mm} \text{for} \hspace{5mm}
l = 0, 1, 2.
\end{equation}
Likewise, for $y_1 \in V^1_1 \cap Q_n$, we have $y_1 + s_0 \in V^2_1 \cap Q_n$, and
\begin{equation}\label{intproof1eqn2}
u_n(y_1 +s_0) = u_n(y_1)+\sigma^0_n.
\end{equation}
Since $\sigma^0_n > 0$, we have the following two possibilities:
\begin{enumerate}[label=(\roman*)]
\item $u_n(\hat x) -\hat x < 0$, and $\hat x$ is contained in $U^1_1 \cup V^1_1$; or
\item $u_n(\hat x) -\hat x> 0$, and  $\hat x$ is contained in $V^2_1 \cup U^3_1$.
\end{enumerate}

\begin{figure}[h]
\centering
\includegraphics[scale=0.4]{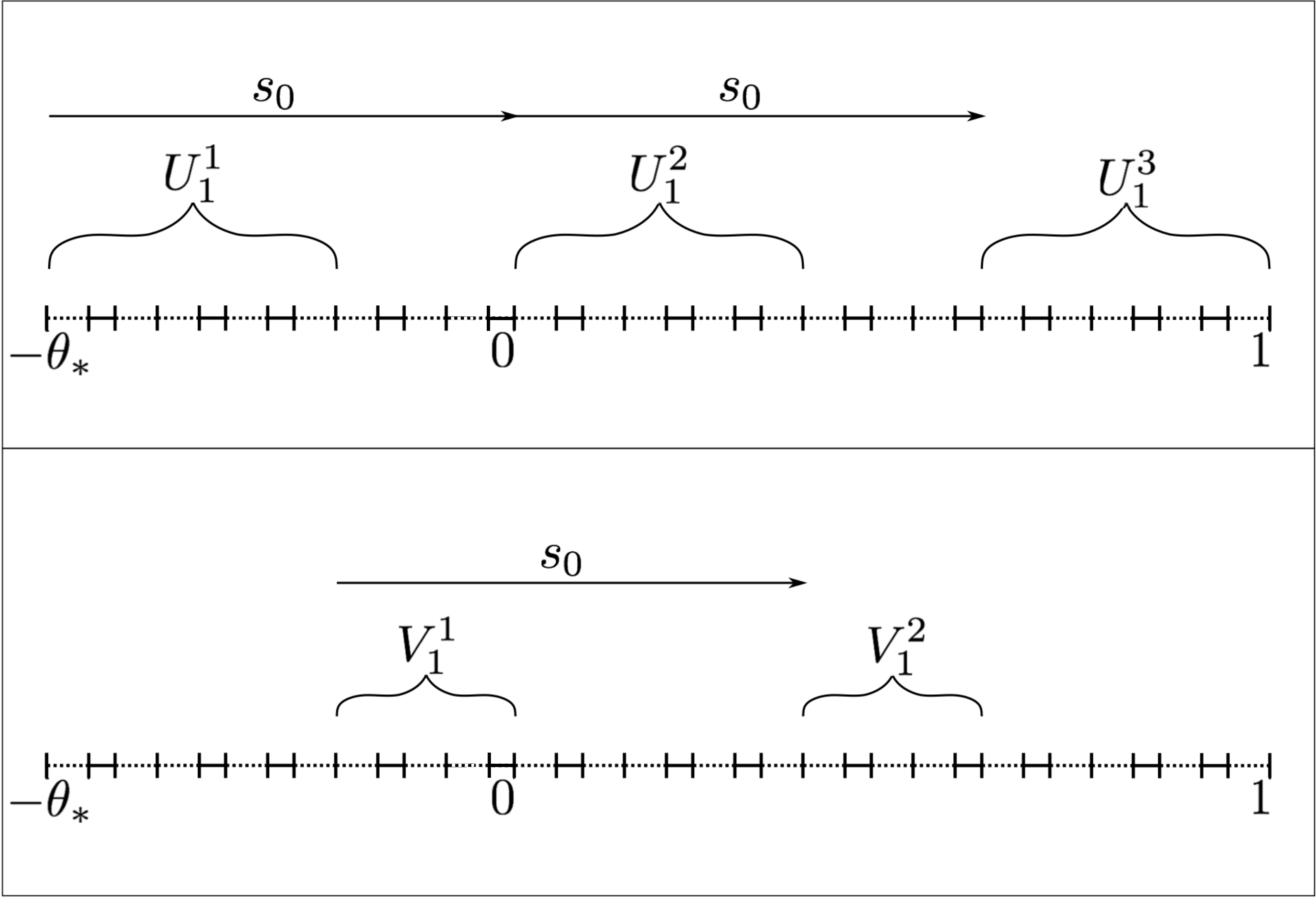}
\caption{Illustration of \eqref{intproof1eqn1} (top) and \eqref{intproof1eqn2} (bottom).}
\label{fig:intproof1}
\end{figure}

Assume case (i). Listing from left to right, the element $U^1_1$ consists of $U^1_2 \in \mathcal{P}_2$, $V_2 \in \mathcal{Q}_2$, and $U^2_2 \in \mathcal{P}_2$. For $x_2 \in (U^1_2 \cup V_2) \cap Q_n$, we have $x_2 + t_1 \in V^1_1 \cap Q_n$, and
\begin{equation}\label{intproof2eqn1}
u_n(x_2 + t_1) = u_n(x_2)+\tau^1_n.
\end{equation}
Moreover, for $y_2 \in U^2_2 \cap Q_n$, we have $y_2 + t_2 \in V^1_1\cap Q_n$, and
\begin{equation}\label{intproof2eqn2}
u_n(y_2 + t_2) = u_n(y_2)+\tau^2_n.
\end{equation}
Since $\tau^1_n, \tau^2_n < 0$, it follows that if $u_n(\hat x) -\hat x < 0$, then $\hat{x} \in V^1_1$. Using a similar argument and proceeding inductively, we see that $\hat{x}$ is contained in $V_{n-1} = [s_{n-1}, 0] \in \mathcal{Q}_{n-1}$. In fact, $\hat{x}$ must be equal to the left endpoint $s_n$ of the unique element of $\mathcal{Q}_n$ contained in $V_{n-1}$. Thus
\begin{displaymath}
|u_n(\hat{x}) - \hat{x}| = |\sigma^{n-1}_n + t_n - \sigma^0_n| < 2t_n = 2\theta_*^{-1}s_n.
\end{displaymath}
The desired estimate follows.

\begin{figure}[h]
\centering
\includegraphics[scale=0.4]{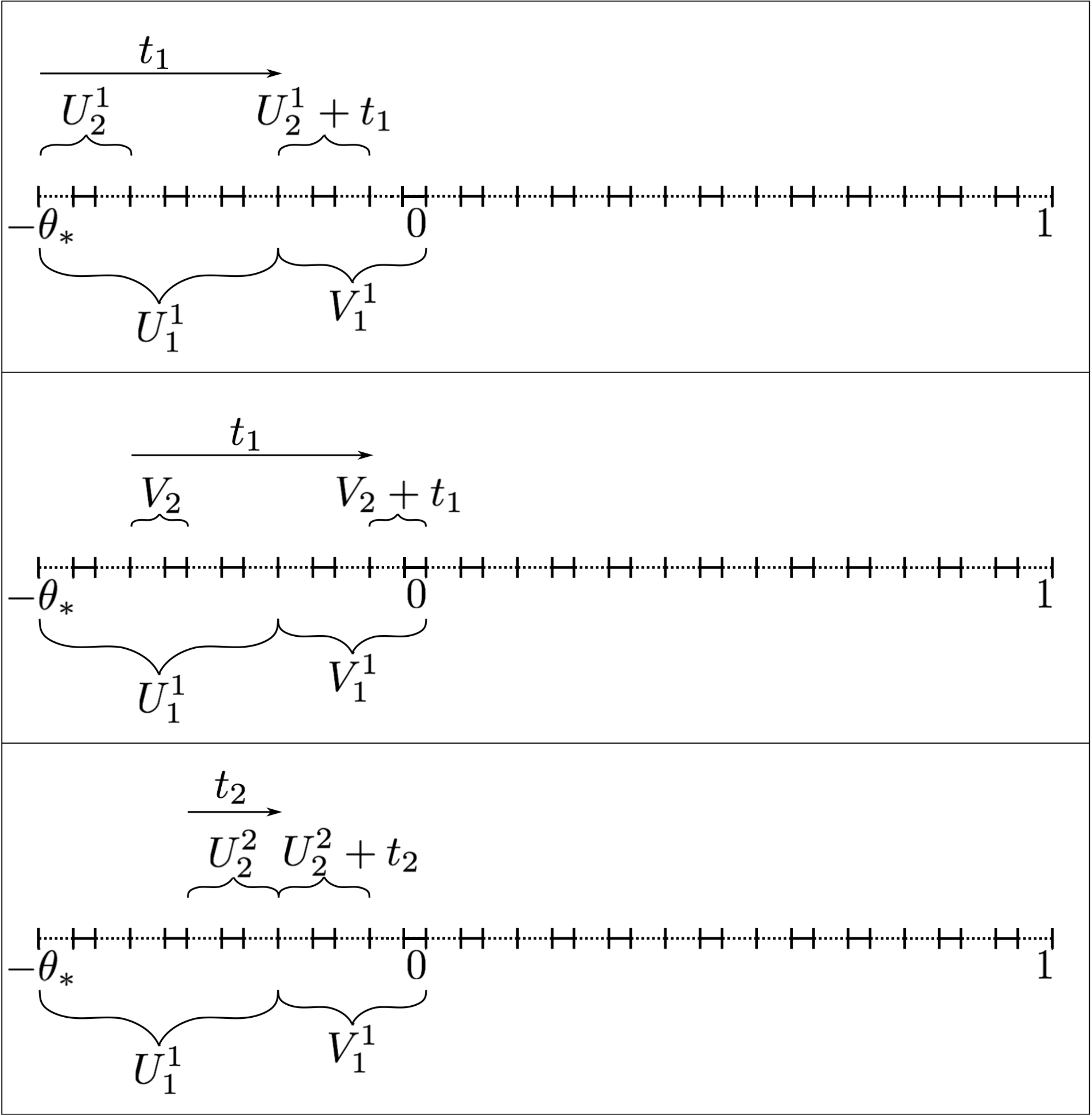}
\caption{Illustration of \eqref{intproof2eqn1} (top and middle) and \eqref{intproof2eqn2} (bottom).}
\label{fig:intproof2}
\end{figure}

Now, assume case (ii). Arguing similarly to above and proceeding inductively, we see that $\hat{x}$ is contained in $J_{n-1} = [1-t_{n-1}, 1] \in \mathcal{P}_{n-1}$. In fact, $\hat{x}$ must be equal to the right endpoint $1-t_n$ of the unique element $I_n$ of $\mathcal{Q}_n$ contained in $J_{n-1}$. Thus
\begin{displaymath}
|u_n(\hat{x}) - \hat{x}| = |1-(1-t_n)| = t_n = \theta_*^{-1}s_n.
\end{displaymath}
The desired estimate follows.
\end{proof}

\begin{figure}[h]
\centering
\includegraphics[scale=0.4]{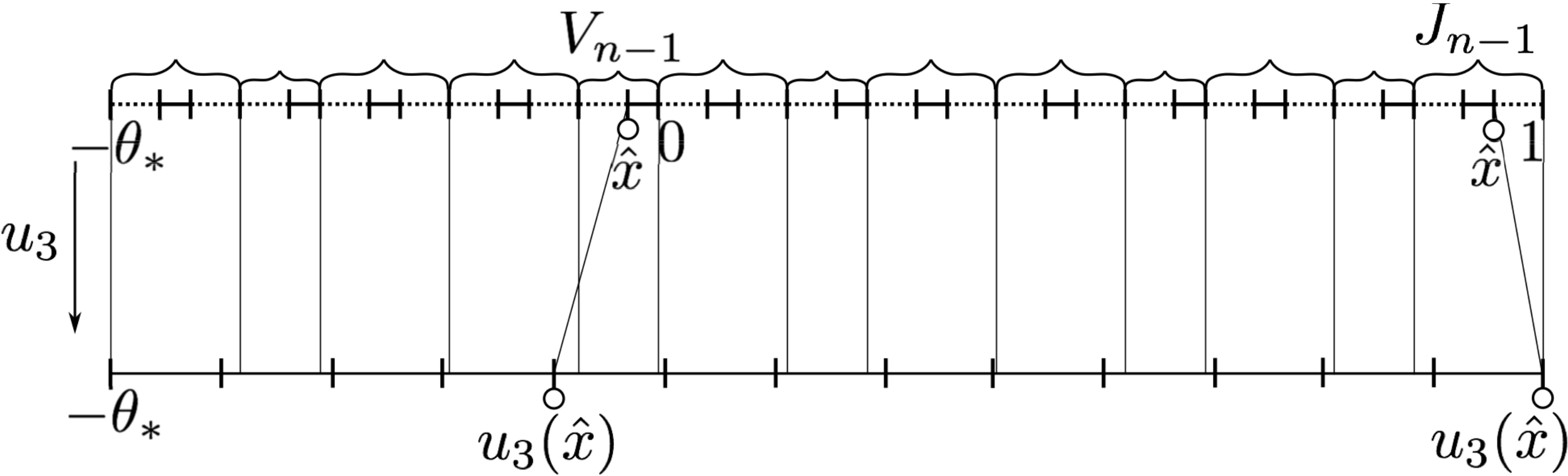}
\caption{Maximal displacements under $u_3$. The point which moves maximally to the left is in $V_{n-1}$ and the point which moves maximally to the right is in $J_{n-1}$.}
\label{fig:intproof2}
\end{figure}

\section{The Renormalization Arc}\label{sec:arc}

Consider the renormalization microscope maps $\Phi^n_k : \Omega \cup \Gamma \to \Omega$ given in \eqref{eq:composition}. Let
\begin{equation}\label{eq:initial bead}
\Omega^n_k := \Phi^n_k(\Omega)
\hspace{5mm} \text{and} \hspace{5mm}
\Gamma^n_k := \Phi^n_k(\Gamma).
\end{equation}

\begin{lem}\label{fixed point}
For all $n \geq 0$, we have
\begin{displaymath}
\Phi_n((1,0)) = (1,0).
\end{displaymath}
Hence, for all $n, k \geq 0$, the domain $\Omega^n_k$ contains the point $(1,0)$.
\end{lem}

\begin{prop}\label{cross}
Let $\lambda_*$ be the universal scaling factor given in Corollary \ref{convergence}. Then for all $n, k \geq 0$, we have
\begin{displaymath}
\text{\emph{diam}}(\Omega^n_k) = O(\lambda_*^n)
\hspace{5mm} \text{and} \hspace{5mm}
\text{\emph{diam}}(\Gamma^n_k) = O(\lambda_*^n).
\end{displaymath}
Consequently, we have
\begin{displaymath}
\bigcap_{n=0}^\infty \Omega^n_k \cup \Gamma^n_k= (1,0).
\end{displaymath}
\end{prop}

\begin{figure}[h]
\centering
\includegraphics[scale=0.25]{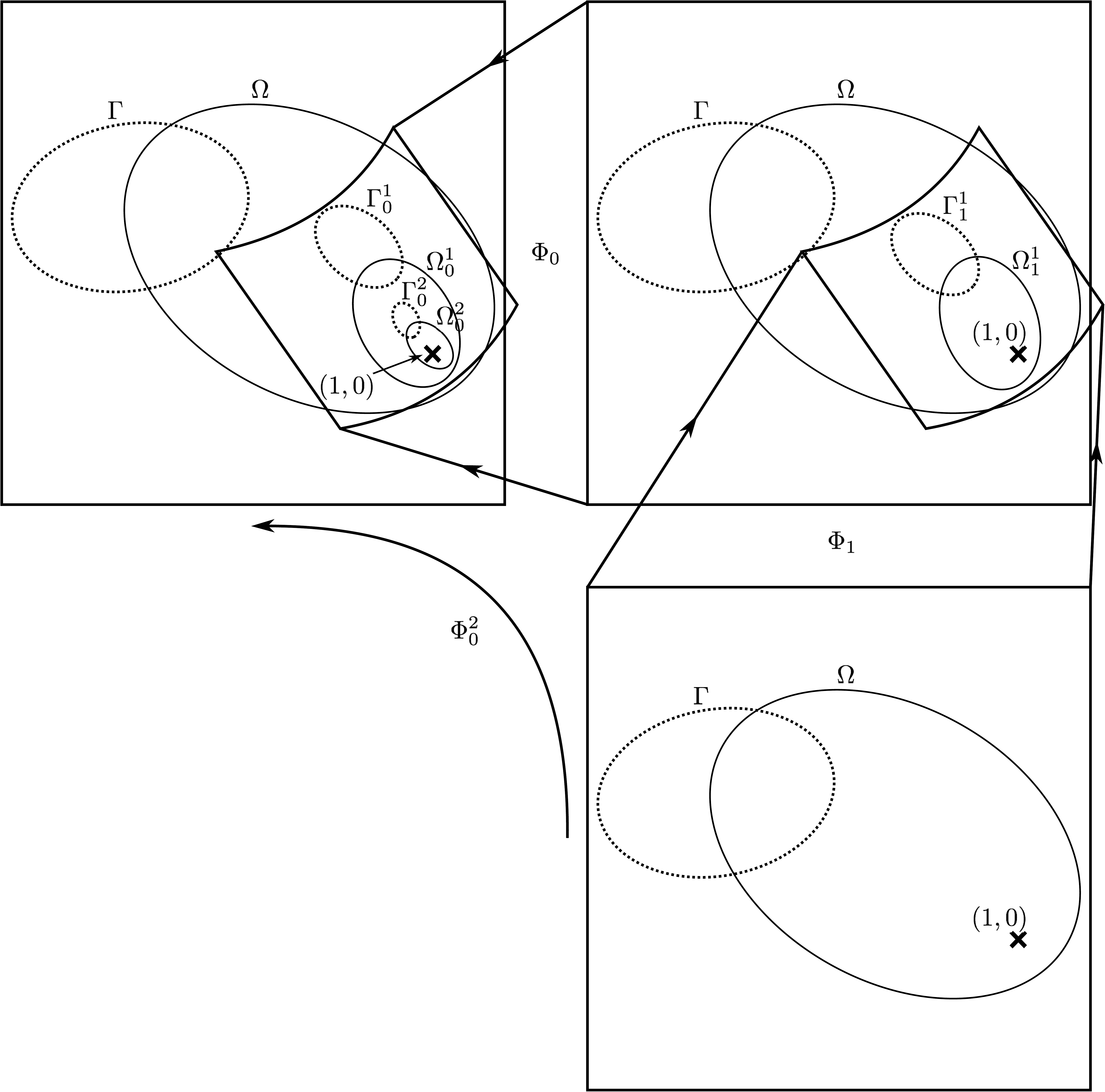}
\caption{The renormalization microscope map $\Phi^2_0$ obtained by composing the non-linear changes of coordinates $\Phi_0$ and $\Phi_1$. We have $\Omega^1_0 = \Phi_0(\Omega)$, $\Gamma^1_0 = \Phi_0(\Gamma)$, $\Omega^2_0 =\Phi^2_0(\Omega)$, $\Gamma^2_0 = \Phi^2_0(\Gamma)$, and $(1,0)=\Phi_0((1,0)) = \Phi^2_0((1,0))$.}
\label{fig:microscope}
\end{figure}

\begin{notn}\label{pre-renormalization}
We denote by
\begin{displaymath}
p\Sigma_n = (pA_n, pB_n)
\hspace{5mm} \text{for} \hspace{5mm}
n \in \mathbb{N}
\end{displaymath}
the sequence of pairs of iterates of $\Sigma=(A,B)$ defined as follows:
\begin{enumerate}[label=(\roman*)]
\item let $p\Sigma_0 := \Sigma$, and
\item for $n \geq 0$, let
\begin{displaymath}
p\Sigma_{n+1} := (pB_n \circ pA_n^2, pB_n \circ pA_n).
\end{displaymath}
\end{enumerate}
\end{notn}

Observe that if
\begin{displaymath}
\Sigma_n = (A_n, B_n)=\mathbf{R}^n(\Sigma)
\end{displaymath}
is the $n$th renormalization of $\Sigma$, then we have
\begin{displaymath}
A_n = (\Phi_0^n)^{-1} \circ pA_n \circ \Phi_0^n
\hspace{5mm} \text{and} \hspace{5mm}
B_n = (\Phi_0^n)^{-1} \circ pB_n \circ \Phi_0^n.
\end{displaymath}

The following statements are analogs of Lemma \ref{rigid bead spread} and Lemma \ref{rigid renormalization spread}.

\begin{lem}\label{bead spread}
Consider the sets $\mathcal{J}_n$ and $\mathcal{I}_n$ of ordered $n$-tuples constructed in Notation \ref{rotationcombinatorics}. For
\begin{displaymath}
\overline{\omega} = (\alpha_{n-1}, \ldots{}, \alpha_0) \in \mathcal{J}_n
\hspace{5mm} \text{and} \hspace{5mm}
\overline{\gamma} = (\beta_{n-1}, \ldots{}, \beta_0) \in \mathcal{I}_n,
\end{displaymath}
denote
\begin{displaymath}
\Sigma^{\overline{\omega}} :=  pA_0^{\alpha_0} \circ \ldots{} \circ pA_{n-1}^{\alpha_{n-1}}
\hspace{5mm} \text{and} \hspace{5mm}
\Sigma^{\overline{\gamma}} :=  pA_0^{\beta_0} \circ \ldots{} \circ pA_{n-1}^{\beta_{n-1}}.
\end{displaymath}
Then $\Sigma^{\overline{\omega}}$ and $\Sigma^{\overline{\gamma}}$ are well-defined on $\Omega^n_0$ and $\Gamma^n_0$ respectively.
\end{lem}

\begin{lem}\label{renormalization spread}
Let
\begin{displaymath}
\overline{\omega}^{\emph{max}}_n := (2, 1, 1, \ldots, 1) \in \mathcal{J}_n
\hspace{5mm} \text{and} \hspace{5mm}
\overline{\gamma}^{\emph{max}}_n := (1, 1, \ldots, 1) \in \mathcal{I}_n.
\end{displaymath}
Then we have
\begin{displaymath}
p\Sigma_n = (pA_n, pB_n)= (B_0 \circ \Sigma^{\overline{\omega}^{\emph{max}}_n}, B_0 \circ \Sigma^{\overline{\gamma}^{\emph{max}}_n}).
\end{displaymath}
\end{lem}

\begin{defn}\label{renormalization arc}
For $n \in \mathbb{N}$, let
\begin{displaymath}
X_n := \bigcup_{\overline{\omega} \in \mathcal{J}_n} \Sigma^{\overline{\omega}}(\Omega^n_0)
\hspace{5mm} \text{and} \hspace{5mm}
Y_n := \bigcup_{\overline{\gamma} \in \mathcal{I}_n} \Sigma^{\overline{\gamma}}(\Gamma^n_0).
\end{displaymath}
The set
\begin{displaymath}
\gamma_\Sigma := \bigcap_{n=1}^\infty X_n \cup Y_n
\end{displaymath}
is called the \emph{renormalization arc} of $\Sigma$.
\end{defn}

\begin{figure}[h]
\centering
\includegraphics[scale=0.3]{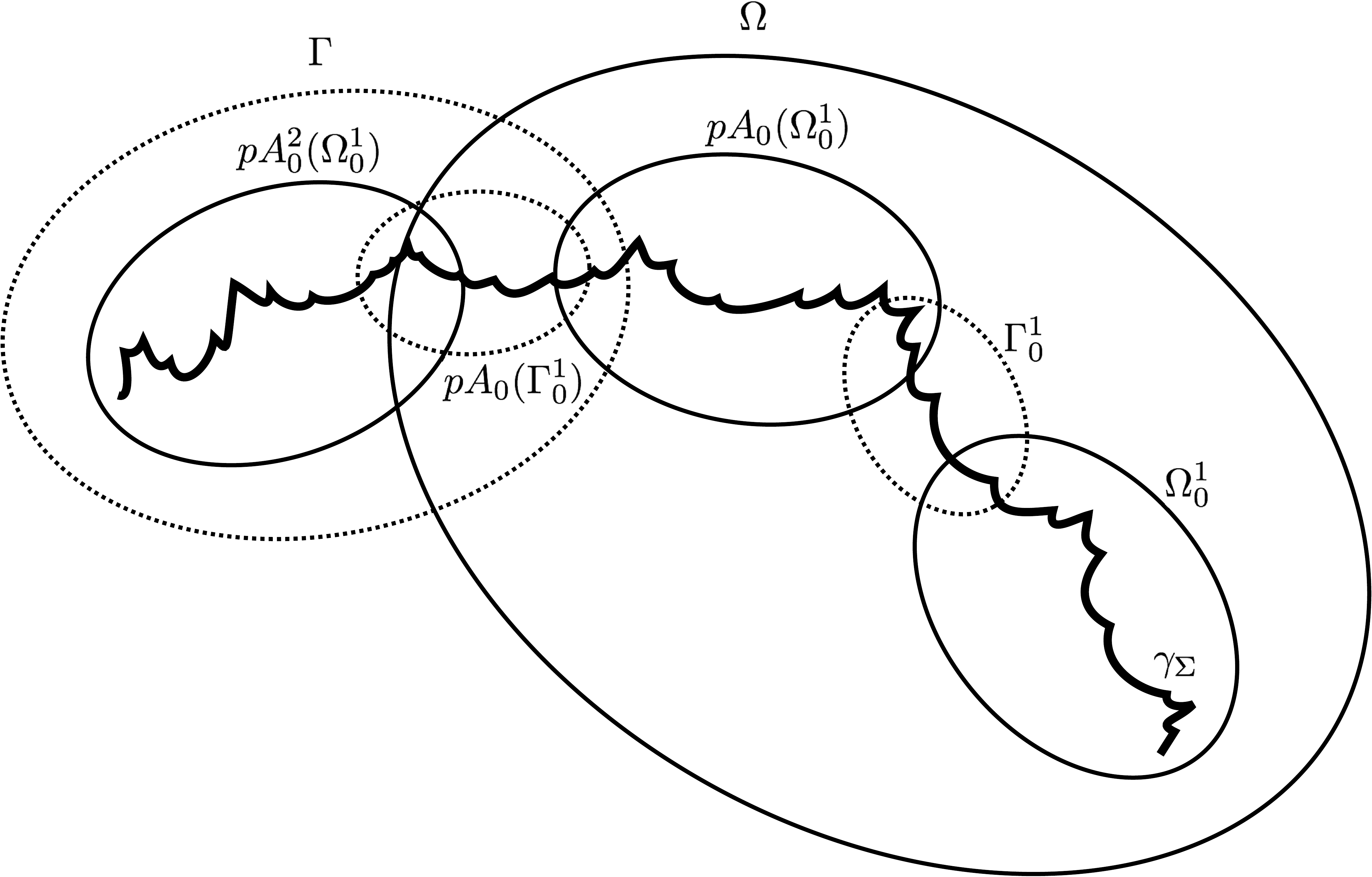}
\caption{The renormalization arc $\gamma_\Sigma$ of $\Sigma$. The open cover $X_1 \cup Y_1$ is shown.}
\label{fig:arc}
\end{figure}

The following theorem justifies the use of the term ``arc'' in Definition \ref{renormalization arc}. It is the counterpart to Proposition 4.2 in \cite{GaRYa}.

\begin{thm}[Continuity of the Siegel Boundary]\label{continuity}
Let $R_0=(S_0|_{J_0}, T_0|_{I_0})$ be the pair representing the rigid rotation of the circle by $\theta_*$ as given by \eqref{eq:rigid pair}. Then there exists a homeomorphism $h : [-\theta_*, 1] \to \gamma_\Sigma$ that conjugates the action of $\Sigma$ and the action of $R_0$.
\end{thm}

\begin{proof}
The proof is identical, {\it mutatis mutandis}, to the proof of Proposition 4.2 in \cite{GaRYa}. For the reader's convenience, we will outline the main ideas.

The renormalization arc $\gamma_{\Sigma_k}$ of the $k$th renormalization of $\Sigma$ maps into $\gamma_\Sigma$ under the microscope map $\Phi^k_0$. For $k$ sufficiently high, $\Sigma_k$ is in a small neighborhood of the renormalization fixed point $\iota(\zeta_*)$. For all such pairs, the iterates $\Sigma_k^{\overline{\omega}}$ for $\overline{\omega} \in \mathcal{J}_n$ and $\Sigma_k^{\overline{\gamma}}$ for $\overline{\gamma} \in \mathcal{I}_n$ have derivatives bounded above by $C\rho^n$ for some uniform constants $C >1$ and $\rho < 1$. It readily follows that the theorem holds for $\gamma_{\Sigma_k}$, and hence, also for $\gamma_\Sigma$.
\end{proof}

Henceforth, we consider the renormalization arc of $\Sigma$ as a continuous curve $\gamma_\Sigma= \gamma_\Sigma(t)$ parameterized by the homeomorphism $h: [-\theta_*, 1] \to \gamma_\Sigma$ given in Theorem \ref{continuity}.

The following theorem is the counterpart to Proposition 4.6 in \cite{GaRYa}. The proof is identical, {\it mutatis mutandis}, and hence, it will be omitted.

\begin{thm} \label{siegel henon}
The pair $\Sigma_{H_{\mu_*,\nu}}$ representing the semi-Siegel H{\'e}non map $H_{\mu_*, \nu}$ given in \eqref{eq:henonpair} is contained in the stable manifold $W^s(\iota(\zeta_*)) \subset \mathcal{D}_2(\Omega, \Gamma, \epsilon)$ of the fixed point $\iota(\zeta_*)$ for $\mathbf{R}$. Moreover, a linear rescaling of the renormalization arc $s(\gamma_{\Sigma_{H_{\mu_*,\nu}}})$ is contained in the boundary of the Siegel disc $\mathcal{D}$ of $H_{\mu_*, \nu}$. In fact, we have
\begin{displaymath}
\partial \mathcal{D} = s(\gamma_{\Sigma_{H_{\mu_*,\nu}}}) \cup H_{\mu_*, \nu} \circ s(\gamma_{\Sigma_{H_{\mu_*,\nu}}}).
\end{displaymath}
\end{thm}

\section{Universality}\label{sec:universality}

Let $\Sigma = (A, B)$ be commuting pair contained in the stable manifold $W^s(\iota(\zeta_*)) \subset \mathcal{D}_2(\Omega, \Gamma, \epsilon)$ of the  2D-renormalization fixed point $\iota(\zeta_*)$. Moreover, assume that there exists a constant $\delta$ such that the following estimates hold:
\begin{equation}\label{eq:jac assumption}
0 \neq \max_{z \in \gamma_\Sigma}\| \text{Jac} \, A(z) \| < \delta
\hspace{5mm} \text{and} \hspace{5mm}
\min_{z \in \gamma_\Sigma}\| \text{Jac} \, B(z) \| > \delta.
\end{equation}
Note that these assumptions hold for the pair $\Sigma_{\mu_*, \nu}$ representing the semi-Siegel H\'enon map $H_{\mu_*, \nu}$ given in \eqref{eq:henonpair}.

By \eqref{eq:jac assumption}, we may choose a branch of the logarithm so that the following complex-valued function is well-defined:
\begin{displaymath}
\tau(t):= \left\{
\begin{array}{ll}
\log \text{Jac} \, A(h(t)) & : \hspace{3mm} 0 < t \leq 1\\
\log \text{Jac} \, B(h(t)) & : \hspace{3mm}  -\theta_* \leq t \leq 0
\end{array}
\right.
\end{displaymath}
where $h : [-\theta_*, 1] \to \gamma_\Sigma$ is the parameterization of the renormalization arc $\gamma_\Sigma$ given in Theorem \ref{continuity}. We define the {\it average Jacobian} of $\Sigma$ to be the following complex number:
\begin{equation}\label{eq:avg jacobian}
b= b_\Sigma := \exp \left( \frac{1}{1+\theta_*}\int_{-\theta_*}^1 \tau(t) \, dt \right).
\end{equation}

Consider the iterate $p\Sigma_n$ of $\Sigma$ given in Notation \ref{pre-renormalization}. Proposition \ref{cross}, Lemma \ref{bead spread} and \ref{renormalization spread}, and standard distortion estimates imply the following:

\begin{lem}\label{distortion lemma}
There exists a uniform constant $\rho<1$ such that
\begin{displaymath}
\frac{\text{\emph{Jac}} \, pB_n(z_1)}{\text{\emph{Jac}} \, pB_n(z_2)}=1 +O(\rho^n),
\end{displaymath}
for any $z_1, z_2 \in \Gamma^n_0$.
\end{lem}

\begin{prop}\label{jac estimate}
Let $\rho<1$ be as in Lemma \ref{distortion lemma}. Then we have
\begin{displaymath}
J_n(z):=\text{\emph{Jac}} \, pB_n(z) = e^{c_n}b^{q_{2n}}(1+O(\rho^n))
\hspace{5mm} \text{for} \hspace{5mm}
z \in \Gamma^n_0,
\end{displaymath}
where $q_{2n} = |\mathcal{I}_n|$, and $c_n \in \mathbb{C}$ has a uniform upper bound.
\end{prop}

\begin{proof}
By Proposition \ref{int prop}, we have
\begin{displaymath}
\log b = \frac{1}{1+\theta_*}\int_{-\theta_*}^1 \tau(t)\, dt = \frac{1}{q_{2n} s_n} \int_{Q_n} \tau(t) \, dt +O(s_n).
\end{displaymath}
Now, there exists a point $x$ in the interval $I_n := [1-t_n-s_n, 1-t_n]$ (see \eqref{base dynamic interval}) such that for
\begin{displaymath}
w := h(x) \in \Gamma^n_0 \cap \gamma_\Sigma,
\end{displaymath}
we have
\begin{displaymath}
\int_{Q_n} \tau(t) \, dt = \int_{h(I_n)} \log \text{Jac} \, pB_n(z) \, dz = s_n \log \text{Jac} \, pB_n(w).
\end{displaymath}
Hence,
\begin{displaymath}
q_{2n}\log b = \log \text{Jac} \, pB_{n}(w) +O(q_{2n} s_n).
\end{displaymath}
Observe that
\begin{displaymath}
q_{2n}s_n < 1+\theta_*.
\end{displaymath}
The result now follows from Lemma \ref{distortion lemma}.
\end{proof}

Set
\begin{displaymath}
\Sigma_n = (A_n, B_n) := \mathbf{R}^n(\Sigma),
\end{displaymath}
where
\begin{displaymath}
A_n(x,y) = \begin{bmatrix}
a_n(x,y) \\
h_n(x,y)
\end{bmatrix}
\hspace{5mm} \text{and} \hspace{5mm}
B_n(x,y) = \begin{bmatrix}
b_n(x,y) \\
x
\end{bmatrix}.
\end{displaymath}
Let
\begin{displaymath}
\eta_n(x) := a_n(x, 0)
\hspace{5mm} , \hspace{5mm}
\xi_n(x) := b_n(x,0)
\hspace{5mm} \text{and} \hspace{5mm}
\zeta_n := (\eta_n, \xi_n).
\end{displaymath}
By Theorem \ref{renormalization hyperbolicity}, we know that the renormalization sequence $\Sigma_{n+1}$ approaches the sequence of embeddings $\iota(\zeta_n)$ super-exponentially fast. The following result, which is central to this paper, states that during in this process, the renormalization sequence uniformizes to a certain two-dimensional {\it universal} form.

\begin{thm}[Universality]\label{universality}
For some $\rho <1$, we have
\begin{displaymath}
B_n(x,y) = \begin{bmatrix}
\xi_n(x) + e^{c_n} b^{q_{2n}} \, \alpha(x) \, y (1+O(\rho^{n}))\\
x
\end{bmatrix},
\end{displaymath}
where $b$ is the average Jacobian of $\Sigma$; the number $q_{2n} \in \mathbb{N}$ is an element of the Fibonacci sequence given in Lemma \ref{closestreturns}; the number $c_n \in \mathbb{C}$  has a uniform upper bound; and $\alpha(x)$ is a universal function that is uniformly bounded away from $0$ and $\infty$, and has a uniformly bounded derivative and distortion.
\end{thm}

\begin{figure}[h]
\centering
\includegraphics[scale=0.4]{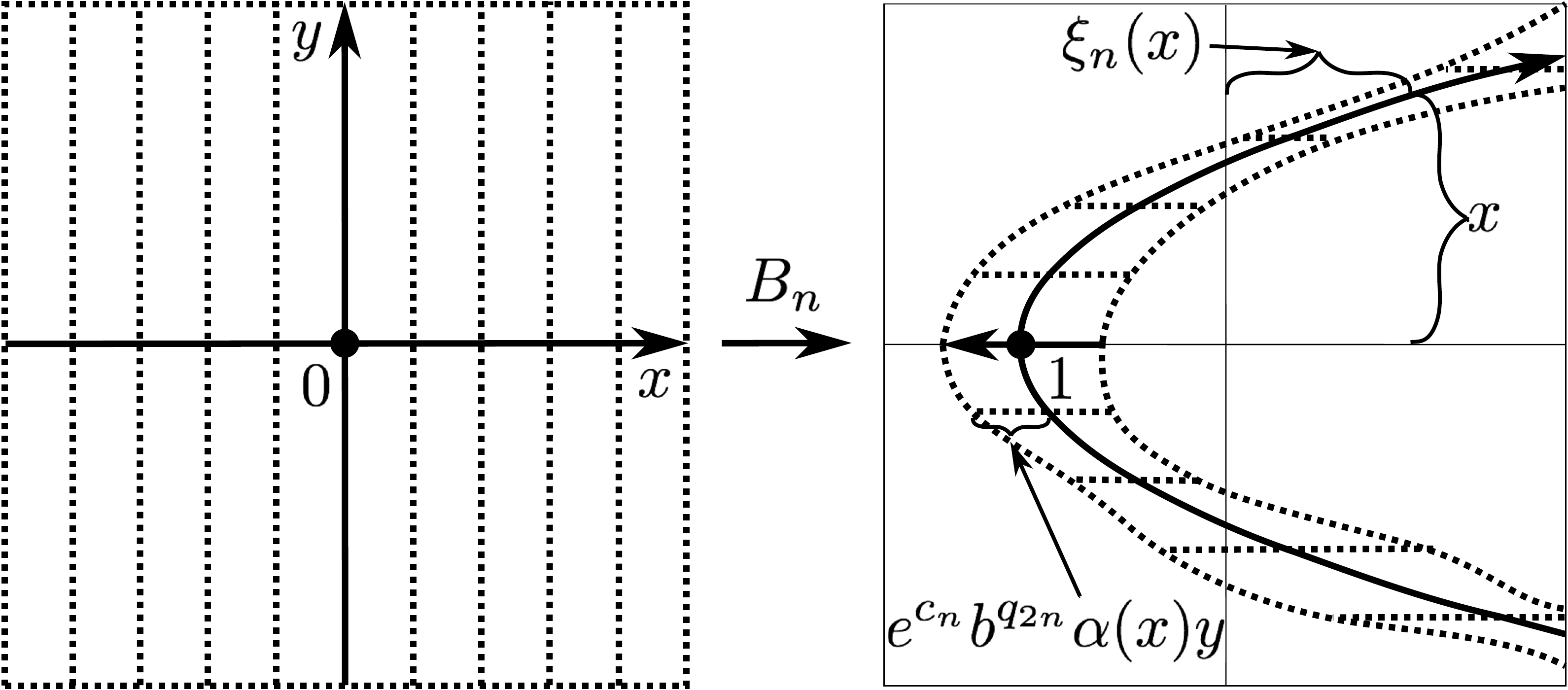}
\caption{The universal 2D form of the $n$th renormalization $\Sigma_n=(A_n, B_n)$ of the commuting pair $\Sigma=(A,B)$. Unlike in Figure \ref{fig:henon}, the vertical lines are not scaled by the same constant. However, the scaling factor $e^{c_n}b^{q_{2n}}\alpha(x)$ is bounded away from $0$ and $\infty$, and has bounded distortion.}
\label{fig:universality}
\end{figure}

\begin{proof}
By definition, we have
\begin{displaymath}
B_n(x,y) = (\Phi_0^n)^{-1} \circ pB_{n} \circ \Phi_0^n(x,y).
\end{displaymath}
Hence,
\begin{equation}\label{eq:jacobian in proof}
\text{Jac} \, B_n(x,y) = J_{n}(x,y) \frac{\text{Jac} \, \Phi_0^n(x,y)}{\text{Jac} \, \Phi_0^n( B_n(x,y))},
\end{equation}
where $J_{n}$ is the Jacobian of $pB_{n}$ given in Proposition \ref{jac estimate}. By Corollary \ref{composition convergence}, we have
\begin{displaymath}
\frac{\text{Jac} \, \Phi_0^n(x,y)}{\text{Jac} \, \Phi_0^n( B_n(x,y))} \to \frac{u_*'(x)}{u_*'(\xi_*(x))} =: \alpha(x)
\hspace{5mm} \text{as} \hspace{5mm}
n \to \infty.
\end{displaymath}
Note that the convergence is geometric, and that $\alpha$ has the properties claimed in the theorem.

Now, write
\begin{displaymath}
B_n(x,y) = \begin{bmatrix}
\xi_n(x)+ E_n(x,y)\\
x
\end{bmatrix}
\end{displaymath}
where $E_n$ is undetermined. Since
\begin{displaymath}
\partial_y E_n(x,y) = \text{Jac} \, B_n(x,y),
\end{displaymath}
plugging in \eqref{eq:jacobian in proof} and integrating both sides, we obtain the desired formula.
\end{proof}

\section{Non-rigidity}\label{sec:nonrigidity}

As an application of the Universality Theorem obtained in Section \ref{sec:universality}, we show that two commuting pairs cannot be $C^1$-conjugate on their respective renormalization arcs if their average Jacobians differ in absolute value. Together with Theorem \ref{siegel henon}, this implies the non-rigidity theorem stated in Section \ref{sec:introduction}. Our proof is similar to the one given in \cite{dCLM} that shows non-rigidity of the invariant Cantor set for period-doubling renormalization.

Consider the non-linear changes of coordinates $\Phi_k$ given in \eqref{eq:change of coordinates}, and their compositions $\Phi_k^n$ given in \eqref{eq:composition}. Denote
\begin{displaymath}
D_k:= D_{(1,0)}\Phi_k
\hspace{5mm} \text{and} \hspace{5mm}
D^k_n := D_{(1,0)}\Phi^k_n.
\end{displaymath}
By Lemma \ref{fixed point}, we have
\begin{displaymath}
D^0_k = \text{Id}
\hspace{5mm} \text{and} \hspace{5mm}
D^n_k = D_k \cdot D_{k+1} \cdot \ldots{} \cdot D_{k+n-1}
\hspace{5mm} \text{for} \hspace{5mm}
n \geq 1.
\end{displaymath}

\begin{prop}\label{tilt derivative}
Let
\begin{displaymath}
D_k = \begin{bmatrix}
1 & s_k \\
0 & 1
\end{bmatrix}
\begin{bmatrix}
u_k & 0 \\
0 & v_k
\end{bmatrix}.
\end{displaymath}
Then there exist constants $\rho <1$ and $K >1$ such that the following estimates hold for all $k \geq 0$:
\begin{enumerate}[label=(\roman*)]
\item $u_k = \lambda_*^2(1+O(\rho^k))$,
\item $v_k = \lambda_*(1+O(\rho^k))$, and
\item $|s_k| \asymp |b|^{q_{2k}}$, where $b$ is the average Jacobian of $\Sigma$ defined in \eqref{eq:avg jacobian}.
\end{enumerate}
Consequently, we have
\begin{displaymath}
D_k^n =  \begin{bmatrix}
1 & s_k \\
0 & 1
\end{bmatrix}
\begin{bmatrix}
\lambda_*^{2n} & 0 \\
0 & \lambda_*^n
\end{bmatrix}(1+O(\rho^k))
\hspace{5mm} \text{for} \hspace{5mm}
n \geq 1.
\end{displaymath}
\end{prop}

\begin{proof}
Consider the renormalization sequence
\begin{displaymath}
 \Sigma_k = (A_k, B_k):=\mathbf{R}^k(\Sigma),
\end{displaymath}
where
\begin{displaymath}
A_k(x,y) = \begin{bmatrix}
a_k(x,y) \\
h_k(x,y)
\end{bmatrix}
\hspace{5mm} \text{and} \hspace{5mm}
B_k(x,y) = \begin{bmatrix}
b_k(x,y) \\
x
\end{bmatrix}.
\end{displaymath}
Let
\begin{displaymath}
\eta_k(x) := a_k(x, 0)
\hspace{5mm} \text{and} \hspace{5mm}
\xi_k(x) := b_k(x,0).
\end{displaymath}

Recall that
\begin{displaymath}
A_{k+1} = \Phi_k^{-1} \circ B_k \circ A_k^2 \circ \Phi_k
\hspace{5mm} \text{and} \hspace{5mm}
B_{k+1} = \Phi_k^{-1} \circ B_k \circ A_k \circ \Phi_k.
\end{displaymath}
Since $\Sigma=(A,B)$ is a commuting pair, we have
\begin{equation}\label{other pair uniformization}
A_{k+1} = \Phi_k^{-1} \circ A_k \circ B_k \circ A_k \circ \Phi_k= \Phi_k^{-1} \circ A_k \circ \Phi_k \circ B_{k+1}.
\end{equation}

Let 
\begin{displaymath}
\tilde{A}_k(x,y)= \begin{bmatrix}
\tilde{a}_k(x,y)\\
\tilde{h}_k(x,y)
\end{bmatrix}
:= \Phi_k^{-1} \circ A_k \circ \Phi_k.
\end{displaymath}
It is easy to check that
\begin{displaymath}
\|\tilde{a}_k - a_k\| = O(\epsilon^{2^{k-1}}),
\hspace{5mm} \text{and} \hspace{5mm}
\|\tilde{h}_k\| = O(\|h_k\|),
\end{displaymath}
where $\epsilon$ is given in \eqref{eq:embedding estimate}.

By Theorem \ref{universality} and \eqref{other pair uniformization}, we have
\begin{align}
a_{k+1}(x,y) &= \tilde{a}_k(\xi_{k+1}(x) + e^{c_{k+1}} b^{q_{2k+2}} \, \alpha(x) \, y (1+O(\rho^{k+1})), x)\nonumber \\
&=\hat \eta_{k+1} \circ \xi_{k+1}(x) + b^{q_{2k+2}} \, \hat \alpha_{k+1}(x) \, y (1+O(\rho^{k+1})),
\label{A universality}
\end{align}
where
\begin{displaymath}
\|\hat{\eta}_{k+1} - \eta_{k+1}\| = O(\epsilon^{2^{k-1}}),
\end{displaymath}
and $\hat \alpha_{k+1}$ is uniformly bounded away from $0$ and $\infty$.

Recall the definition of $\Phi_k$ given in \eqref{eq:change of coordinates}. Using Corollary \ref{convergence} and Proposition \ref{kappa}, the result now follows by a straightforward computation.
\end{proof}

\begin{thm}[Non-rigidity]\label{non-rigidity}
Let $\Sigma=(A,B)$ and $\tilde{\Sigma}=(\tilde{A},\tilde{B})$ be commuting pairs contained in the stable manifold $W^s(\iota(\zeta_*)) \subset \mathcal{D}_2(\Omega, \Gamma, \epsilon)$ of the  2D-renormalization fixed point $\iota(\zeta_*)$. Furthermore, assume that $\Sigma$ and $\tilde{\Sigma}$ both satisfy \eqref{eq:jac assumption} for some $\delta, \tilde{\delta}>0$, so that their respective average Jacobians $b$ and $\tilde{b}$ are well-defined. Let $f : \gamma_{\Sigma} \to \gamma_{\tilde{\Sigma}}$ be a homeomorphism which conjugates the action of $\Sigma$ and $\tilde{\Sigma}$. Then the H\"older exponent of $f$ is at most $\frac{1}{2}(1+\ln|b|/\ln|\tilde{b}|)$ (and in particular, cannot be $C^1$).
\end{thm}

\begin{proof}
For brevity, we will only define the notations for $\Sigma$. The corresponding objects for $\tilde{\Sigma}$ will be marked with the tilde.

Choose $k$ sufficiently large so that
\begin{displaymath}
|\tilde{b}|^{q_{2k}}<<|b|^{q_{2k}}.
\end{displaymath}
Next, choose $n \geq 0$ so that
\begin{equation}\label{eq:avg jacobian bounds}
\lambda_*^{n+1} < |\tilde{b}|^{q_{2k}} < \lambda_*^n << |b|^{q_{2k}}.
\end{equation}

For the proof, we work in three different scales: in the scale of $\Sigma=(A,B)$, of $\Sigma_k=(A_k, B_k)$ and of $\Sigma_{k+n} = (A_{k+n}, B_{k+n})$ (see Figure \ref{fig:nonrigidity}). First, in the scale of $\Sigma_{k+n}$, let
\begin{displaymath}
c_{k+n} := A_{k+n}((1,0)).
\end{displaymath}
Then, in the scale of $\Sigma_k$, let
\begin{displaymath}
c_k^n := \Phi_k^n(c_{k+n})
\hspace{5mm} , \hspace{5mm}
z_k^n := A_k(c_k^n)
\hspace{5mm} \text{and} \hspace{5mm}
w_k := A_k((1,0)).
\end{displaymath}
Finally, in the scale of $\Sigma$, let
\begin{displaymath}
Z_k^n := \Phi_0^k(z_k^n)
\hspace{5mm} \text{and} \hspace{5mm}
W_k := \Phi_0^k(w_k).
\end{displaymath}

Consider the distance between the following pairs of points:
\begin{enumerate}
\item $c_k^n$ and $(1,0)$,
\item $z_k^n$ and $w_k$, and
\item $Z_k^n$ and $W_k$.
\end{enumerate}
Let $\Delta^x_i$, $\Delta^y_i$ and $\Delta_i$ with $i = 1, 2, 3$ denote the horizontal, vertical and Euclidean distance between these pairs of points respectively. 

\begin{figure}[h]
\centering
\includegraphics[scale=0.41]{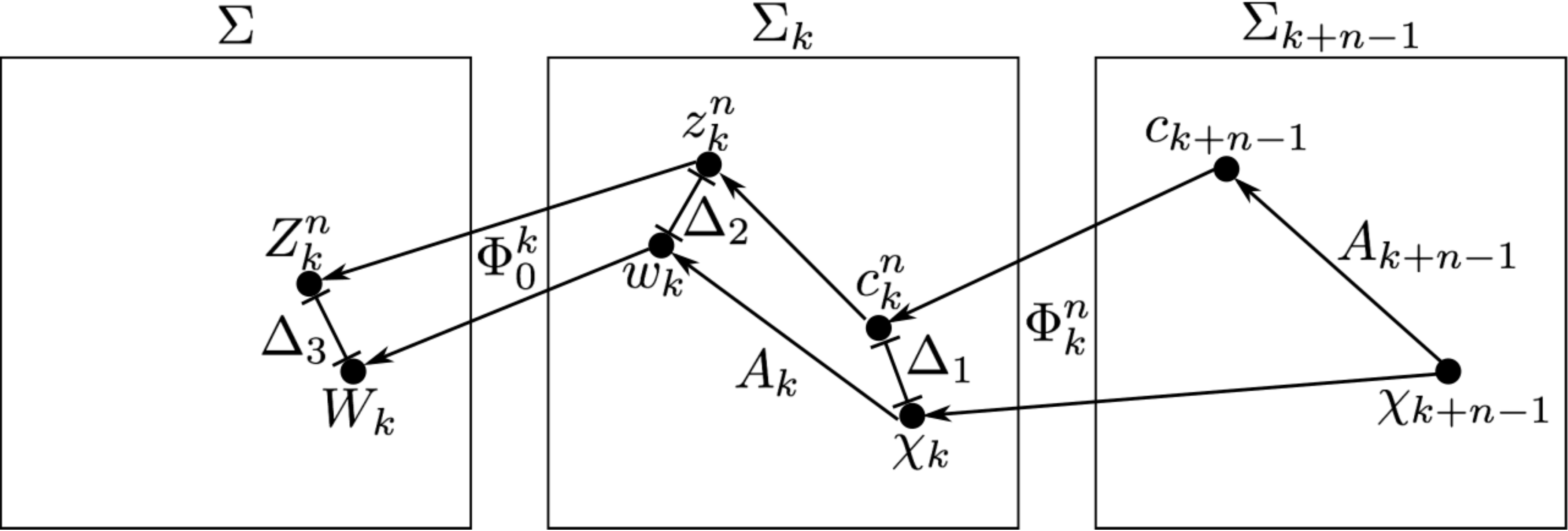}
\caption{Distances $\Delta_1$, $\Delta_2$, and $\Delta_3$.}
\label{fig:nonrigidity}
\end{figure}

By Proposition \ref{tilt derivative}, we have
\begin{displaymath}
\Delta^y_1 \asymp |\lambda_*|^n,
\end{displaymath}
and for some uniform constant $C >0$:
\begin{displaymath}
\Delta^x_1 > 2C(|b|^{q_{2k}}|\lambda_*|^n - |\lambda_*|^{2n}) > C |b|^{q_{2k}} |\lambda_*|^n,
\end{displaymath}
where in the last inequality we used \eqref{eq:avg jacobian bounds}. Thus, we see that
\begin{displaymath}
\Delta^y_2 > C|b|^{q_{2k}} |\lambda_*|^n.
\end{displaymath}
Again by Proposition \ref{tilt derivative}, we arrive at:
\begin{displaymath}
\Delta_3 \geq \Delta^y_3 > C|b|^{q_{2k}} |\lambda_*|^{n+k}.
\end{displaymath}

Now, consider the corresponding distances for $\tilde{\Sigma}$. Again, we have
\begin{displaymath}
\tilde{\Delta}^y_1 \asymp |\lambda_*|^n.
\end{displaymath}
However, by \eqref{eq:avg jacobian bounds} we see that
\begin{displaymath}
\tilde{\Delta}^x_1 = O(|\tilde{b}|^{q_{2k}}|\lambda_*|^n + |\lambda_*|^{2n}) = O(|\lambda_*|^{2n}).
\end{displaymath}
By \eqref{A universality} and \eqref{eq:avg jacobian bounds}, we obtain
\begin{displaymath}
\tilde{\Delta}^x_2 = O(\tilde{\Delta}^x_1 + |\tilde{b}|^{q_{2k}}\tilde{\Delta}^y_1) = O(|\lambda_*|^{2n}) = \tilde{\Delta}^y_2.
\end{displaymath}
Lastly, Proposition \ref{tilt derivative} implies that:
\begin{displaymath}
\tilde{\Delta}^x_3 =\tilde{\Delta}^y_3 = O(|\lambda_*|^{2n+k}).
\end{displaymath}
Hence,
\begin{displaymath}
\tilde{\Delta}_3 = O(|\lambda_*|^{2n+k}).
\end{displaymath}

Observe that any H\"older exponent $\alpha$ for a conjugacy $f : \gamma_{\Sigma} \to \gamma_{\tilde{\Sigma}}$ between $\Sigma$ and $\tilde{\Sigma}$ must satisfy
\begin{displaymath}
\Delta_3 \leq C'(\tilde{\Delta}_3)^\alpha.
\end{displaymath}
for some uniform constant $C'>1$. By our estimates above, this means
\begin{displaymath}
|b|^{q_{2k}} |\tilde{b}|^{q_{2k}} |\lambda_*|^k < |b|^{q_{2k}} |\lambda_*|^{n+k} \leq C'(|\lambda_*|^{2n+k})^\alpha < C'(|\lambda_*|^{k-2} |\tilde{b}|^{q_{2k}} |\tilde{b}|^{q_{2k}})^\alpha.
\end{displaymath}
The theorem follows.
\end{proof}

\end{document}